\RequirePackage{fix-cm}

\documentclass[smallextended,nospthms]{svjour3}       
\smartqed  
\usepackage{amsmath,amsthm,amssymb,amscd,latexsym,hyperref,graphicx, pdfsync,xcolor}

\newtheorem{theorem}{Theorem}[section]
\newtheorem{lemma}[theorem]{Lemma}
\newtheorem{corollary}[theorem]{Corollary}
\newtheorem{prop}[theorem]{Proposition}
\newtheorem{defi}{Definition}

\newtheorem{nota}[theorem]{Notations}
\newtheorem{exa}{Example}

\def \build#1#2#3{\mathrel{\mathop{\kern 0pt#1}\limits_{#2}^{#3}}}

\newcommand{\Pn}{\mathbb P_n}
\newcommand{\B}{\mathbb B}
\newcommand{\Bn}{\mathbb B_n}
\newcommand{\Pnt}{\mathbb L_n}
\newcommand{\Pnb}{\mathbb G_{n}}
\newcommand{\pnh}{\widehat{\mathbb G}_n}
\renewcommand{\P}{\mathbb P}

\newcommand{\R}{\mathbb R}

\newcommand{\N}{\mathbb N}
\newcommand{\ac}[1]{\left\{#1\right\}}
\newcommand{\norm}[1]{\left\Vert#1\right\Vert}
\newcommand{\pa}[1]{\left(#1\right)}

\newcommand{\ceil}[1]{\left\lceil#1\right\rceil}
\newcommand{\abs}[1]{\left|#1\right|}

\newcommand{\pr}[1]{\mathbb{P}\left(#1\right)}
\newcommand{\esp}[2]{\mathbb{E}_{#1}\left[#2\right]}
\newcommand{\bigO}[1]{\mathcal O\pa{#1}}

\newcommand{\fonctelle}[1]{\langle#1\rangle}

\newdimen\AAdi%
\newbox\AAbo%
\def\AAk#1#2{\setbox\AAbo=\hbox{#2}\AAdi=\wd\AAbo\kern#1\AAdi{}}%
\newcommand{\ind}{ 1\hspace{-.55ex}\mbox{l}}

\begin{document}

\title{Asymptotic behavior of some factorizations of random words
}

\titlerunning{Statistics of some factorizations of random words}        

\author{Elahe Zohoorian Azad         \and
        Philippe Chassaing 
}


\institute{E. Zohoorian Azad  \at
              School of mathematics and computer sciences\\
               Damghan University\\
               P.O.Box 36715-364\\
Damghan, Iran 
           \and
           P. Chassaing \at
             Institut Elie Cartan de Lorraine\\
Universit{\'e} de Lorraine\\
Campus Scientifique, BP 239 \\
54506 Vandoeuvre-l{\`e}s-Nancy  Cedex France 
}


\maketitle

\begin{abstract}
In this paper we consider the normalized  lengths of the factors of some factorizations of random words. First, for the \emph{Lyndon factorization} of finite random words with $n$ independent letters drawn from a finite or infinite totally ordered alphabet according to a general probability distribution, we prove that the limit law of the  normalized lengths of the smallest Lyndon factors is a variant of the stickbreaking process. Convergence of the distribution of the lengths of the longest factors to a Poisson-Dirichlet distribution follows. Secondly we consider the \emph{standard factorization} of random  \emph{Lyndon word} : we prove that the distribution of the normalized length of the standard right factor of a random  $n$-letters long Lyndon word, derived from such an alphabet, converges, when $n$ is large, to:
$$\mu(dx)=p_1 \delta_{1}(dx) + (1-p_1) \mathbf{1}_{[0,1)}(x)dx,$$
in which $p_1$ denotes the probability of the smallest letter of the alphabet. 
\end{abstract}
\tableofcontents

\section{Introduction}
\label{intro}
In this paper we address the statistical properties of two well studied factorizations related to Lyndon words: the \emph{Lyndon factorization} of a word, and the \emph{standard factorization} of a Lyndon word. Applications of these two factorizations are discussed in \cite[Section 4.2]{MR2300777}. Let us recall some notations and definitions from \cite{MR1475463,Reutenauer} for readability. For an ordered alphabet $\mathcal{A}=\{\mathtt{a}_1\prec \mathtt{a}_2\prec\dots\}$, finite or infinite, $\mathcal{A}^{n}$ is the set of  $n$-letters words, and the \emph{language}, i.e. the set of finite words, is
\[\mathcal{A}^{\star}=\ac{\emptyset}\cup\mathcal{A}\cup\mathcal{A}^{2}\cup\mathcal{A}^{3}\cup\dots,\] 
while $\mathcal{A}^{+}$ denotes $\mathcal{A}^{\star}\backslash \ac{\emptyset}$. The length of a word $\mathtt{w}\in \mathcal{A}^{\star}$ is denoted by $|\mathtt{w}|$. The language $\mathcal{A}^{\star}$ is endowed with an operation, the \emph{concatenation} $\mathtt{uv}=\mathtt{w}$ of two words $\mathtt{u}$ and $\mathtt{v}$, that is also  a \emph{factorization} of $\mathtt{w}$. A word $\mathtt{v}$ is a \emph{factor} of a
word $\mathtt{w}$ if there exists  two other words $\mathtt{p}$ and $\mathtt{s}$, possibly empty, such that $\mathtt{w}=\mathtt{pvs}$. If $\mathtt{p}$ (resp. $\mathtt{s}$) is empty, $\mathtt{v}$ is a prefix (resp. a suffix) of $\mathtt{w}$.   If $|\mathtt{p}|=k-1$ and $|\mathtt{pv}|=\ell$,  $\mathtt{w}_{[k,\ell]}$ denotes the factor  $\mathtt{v}$ of the word   $\mathtt{w}$.  

The total order, $\prec$, on the alphabet $\mathcal{A}$, induces a corresponding \emph{lexicographic order}, again denoted by $\prec$, on  $\mathcal{A}^{+}$: the word $\mathtt{v}$ is smaller than the word $\mathtt{w}$ (for the lexicographic order, $\mathtt{v}\prec \mathtt{w}$) at one of the following conditions: either $\mathtt{v}$ is a proper prefix of $\mathtt{w}$, or there exist words $\mathtt{p}$, $\mathtt{s}_{1}$, $\mathtt{s}_{2}$ in $\mathcal A^{\star}$  and letters $\mathtt{a}\prec \mathtt{b}$ in $\mathcal A$, such that $\mathtt{v}=\mathtt{pas}_{1}$ and $\mathtt{w}=\mathtt{pbs}_{2}$. For any factorization  $\mathtt{w}=\mathtt{uv}$ of $\mathtt{w}$, $\mathtt{vu}$ is called  a rotation of $\mathtt{w}$, and the set $\fonctelle{\mathtt{w}}$ of rotations of $\mathtt{w}$ is called  the \emph{necklace} of $\mathtt{w}$. A word $\mathtt{w}$ is primitive if $|\mathtt{w}|=\#\fonctelle{\mathtt{w}}$. In this case the necklace is said to be \textit{aperiodic}.  A word $\mathtt{v}$ is a factor of a
necklace $\langle \mathtt{w}\rangle$ if $\mathtt{v}$ is a factor of some word
$\mathtt{w'}\in\langle \mathtt{w}\rangle$.

The notion of \emph{Lyndon word} has many equivalent  definitions, to be found, for instance, in \cite{MR1475463}. 
\begin{defi}[Lyndon word]
A word $\mathtt{w}\in\mathcal A^{+}$ is a Lyndon word if one of the 2 following (equivalent) conditions is satisfied :
\begin{enumerate}
\item $\mathtt{w}$ is primitive and is the smallest element of $\fonctelle{\mathtt{w}}$ ; 
\item $\mathtt{w}$ is smaller than any proper suffix of $\mathtt{w}$. 
\end{enumerate}
\end{defi}
\begin{exa}
The word $\mathtt{w}=\mathtt{aabaab}$ is the smallest in its necklace 
\[\fonctelle{\mathtt{w}}=\ac{\mathtt{aabaab},\mathtt{abaaba},\mathtt{baabaa}}\]
but is not Lyndon; $\mathtt{baac}$ is not Lyndon, nor $\mathtt{acba}$ or $\mathtt{cbaa}$, but $\mathtt{aacb}$ is Lyndon. Here is an ordered list of  Lyndon words $\mathtt{w}\in\{\mathtt{a},\mathtt{b}\}\cup\{\mathtt{a},\mathtt{b}\}^2\cup\{\mathtt{a},\mathtt{b}\}^3$:$$\mathtt{a}\prec \mathtt{aab}\prec \mathtt{ab}\prec \mathtt{abb}\prec \mathtt{b}.$$
\end{exa}
Let $\mathfrak{L}$ denote the set of Lyndon words on the alphabet $\mathcal{A}$. Note that $\mathfrak{L}\subset\mathcal{A}^{+}$. A recursive characterization of Lyndon words is as follows:
\begin{defi}[and Proposition]
One-letter words are Lyndon. A word $\mathtt{w}$ with length $n\ge  2$ is a Lyndon word if and only if there exists two Lyndon words $u$ and $v$ such that $\mathtt{w}=\mathtt{uv}$ and $\mathtt{u}\prec \mathtt{v}$. Among such factorizations of a Lyndon word $\mathtt{w}$, the  factorization with the smallest \footnote{''Smallest'' does not mean ''shortest'', actually it means ''longest'' here.} suffix $\mathtt{v}$ is called the \emph{standard factorization}.
\end{defi}
\begin{exa}
$\mathtt{0011}=\mathtt{(001)(1)}=\mathtt{(0)(011)}$ is a Lyndon word with two such factorizations. The latter is the standard factorization. Others examples of standard factorizations: $ \mathtt{aaabaab}=\mathtt{aaab.aab},\  \mathtt{aaababb}=\mathtt{a.aababb},\  \mathtt{aabaabb}=\mathtt{aab.aabb}$. 
\end{exa}

The \textit{standard right factor} $v$ of a several letters-Lyndon word $\mathtt{w}$ is its smallest proper Lyndon suffix, but also its smallest proper suffix. The standard factorization of a Lyndon word is the first step in the
construction of some basis of the free Lie algebra over $\mathcal{A}$, due
to Lyndon~\cite{Lyndon} (see for instance \cite{MR1475463} or
\cite{Reutenauer}). \\

While the standard factorization has always 2 factors and applies to Lyndon words, the Lyndon factorization, useful for instance in data compression, see \cite{DBLP:journals/corr/abs-1201-3077}, has a variable number of factors, and applies to any word in $\mathcal{A}^+$. It is defined as follows : 
\begin{theorem} {\bf and Definition }[Chen, Fox and Lyndon, cf. \cite{MR1475463} Theorem 5.1.5]
\label{lyndon}  Any word $\mathtt{w}\in {\mathcal{A}}^{+}$ has a unique factorization  as a nonincreasing product of Lyndon words, called the \emph{Lyndon factorization} of  $\mathtt{w}$ :
$$ \mathtt{w}=\mathtt{w}_\ell \mathtt{w}_{\ell-1}\dots \mathtt{w}_2 \mathtt{w}_1,\hspace{0.5cm}\mathtt{w}_i\in \mathfrak{L},\hspace{0.5cm}\mathtt{w}_\ell\succeq \mathtt{w}_{\ell-1}\succeq\dots\succeq \mathtt{w}_2\succeq \mathtt{w}_1.$$
\end{theorem}

Here, as opposed to the standard factorization, a Lyndon word  factors trivially (it has only one factor).
\begin{figure}[ht] 
\label{LyndonExample}
\begin{center}
{\includegraphics[width=7cm]{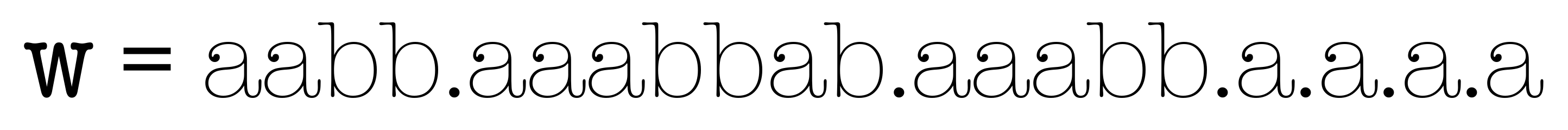}}
\caption{A word $\mathtt{w}$ with 7 factors and a sequence of lengths :} \  \  \  \  \  \  \  
\  \  \ $\rho^{(20)}(\mathtt{w})=\tfrac1{20}\,(1, 1, 1, 1, 5, 7, 4, 0, 0\dots)$
\end{center}
\end{figure}
%
\subsection{Random words and random Lyndon words}

In this paper, we  study the asymptotic behavior of some natural statistics  related to :
\begin{itemize}
\item the Lyndon factorization of a $n$-letters long random word chosen in $\mathcal A^n$ according to the probability distribution $\mathbb P_n$ ;
\item the standard factorization of a $n$-letters long random  Lyndon word  chosen in $\mathfrak L_n$ according to the probability distribution $\mathbb L_n$.
\end{itemize}

Both $\mathbb P_n$ and $\mathbb L_n$ depend on a general probability distribution $(p_{i})_{i\geq 1}$  on the finite or infinite alphabet $\mathcal A=\ac{\mathtt{a}_{1}\prec \mathtt{a}_{2}\prec\dots}$, and we assume, without loss of generality, that $0<p_{1}<1$. On the corresponding language $\mathcal A^\star$, we define the weight $p(\mathtt{w})$ of a word $\mathtt{w}=\mathtt{a}_{\ell_{1}}\mathtt{a}_{\ell_{2}}\dots \mathtt{a}_{\ell_{n}}$ as
\[
p(\mathtt{w}) = p_{\ell_{1}}p_{\ell_{2}}\dots p_{\ell_{n}}.
\]
The weight $p(.)$ defines a probability measure $\mathbb P_{n}$ on the set $\mathcal A^n$,  through
\[\mathbb P_{n}(\{\mathtt{w}\})=p(\mathtt{w}).\]
$\mathcal P_n$ (resp. $\mathcal N_n$, $\mathfrak L_n$) denotes  the set of $n$-letters long  primitive words (resp.  its complement, resp. the set of $n$-letters long Lyndon words). Then we define a probability measure $\Pnt$ on $\mathfrak L_n$, as follows
\[
\Pnt(\{\mathtt{w}\}) = \lambda_{n}p(\mathtt{w}),
\]
 in which $\lambda_{n}=1/\mathbb P_{n}(\mathfrak L_n)=n/\mathbb P_{n}(\mathcal
 P_n)$. The probability measure  $\Pnt$ has a trivial extension to $\mathcal{A}^n$ (setting $\Pnt\pa{\mathfrak L_{n}^c}=0$).

\subsection{Main results}
\label{mainresults}

The sequence $\rho^{(n)}(\mathtt{w})= (\rho_{i,n}(\mathtt{w}))_{i\ge 1}$ of normalized lengths of the Lyndon factors of a word $\mathtt{w} \in \mathcal A^n$, with Lyndon factorization $\mathtt{w}=\mathtt{w}_{\ell} \mathtt{w}_{\ell-1}\dots \mathtt{w}_1$, is defined as follows: 
\[\rho_{i,n}(\mathtt{w})=\begin{cases}
\begin{array}{ll}
\frac {|\mathtt{w}_i|}{n}&\text{if }\ 1\le i\le \ell
\\
0&\text{if }\ i> \ell.
\end{array}
\end{cases}
\]
Thus $\rho_{i,n}(\mathtt{w})$ denotes the normalized length of the $i$-th smallest Lyndon factor of $\mathtt{w}$ (remark that in this paper, when applied to words, e.g. to factors of words and necklaces, the adjectives ``small'' and ``large'' refer to the \textit{lexicographic order} on words, while ``short'' and ``long'' refer to the \textit{size}, or number of letters). Our first result describes  the limit distribution, as $n$ grows, of the sequence $\rho^{(n)}(\mathtt{w})=(\rho_{i,n}(\mathtt{w}))_{i\ge 1}$, seen as a random variable on $\pa{\mathcal A^n,\mathbb P_n}$. We have:

\begin{theorem}
\label{lyndonlengths} For a totally ordered alphabet with
probability distribution $p$ on its letters, $\rho^{(n)}$ converges in law, when $n\rightarrow\infty$, to the random sequence  $\rho= (\rho_{i})_{i\ge 1}$ whose law is defined by the law of $\rho_1$:
$${\mu}(dx)={p_1}\delta_{0}(dx)+ {(1-p_1)} {\mathbf{1}}_{(0,1]}(x)dx,$$
and by the conditional distribution ${\mu}^y$ of $\rho_i$ given $(\rho_1,\rho_2,\dots,\rho_{i-1})$, that only depends on $y=\rho_1+\rho_2+\dots+\rho_{i-1}$, and is as follows :
$${\mu}^y (dx) = \begin{cases}
\begin{array}{lll}
{p_1} \delta_{0}(dx)+
{(1-p_1)} {\mathbf{1}}_{(0,1]}(x)dx&\quad\text{if~}y=0 \\
&\\
 \frac{1}{1-y}{\mathbf{1}}_{(0,1-y]}(x)dx&\quad\text{if~}y>0.
\end{array}
\end{cases}
$$
\end{theorem}
In other words, if we set
$s_{i}=1-(\rho_{1}+\rho_{2}+\dots+\rho_{i})$, then
$s=(s_{i})_{i\ge1}$ is a Markov chain starting from $1$ at time
$0$, with transition probability
$$p(y,dx) = \begin{cases}
\begin{array}{ll}
{p_1} \delta_{1}(dx)+
{(1-p_1)} {\mathbf{1}}_{(0,1]}(x)dx\ ; \; \; \; y=1 \\
&\\
 \frac{1}{y}{\mathbf{1}}_{(0,y]}(x)dx\ ; \; \; \; y<1.
\end{array}
\end{cases}
$$
The process $s$ is a variant of the \emph{stickbreaking process} \cite{MR2615013,MR1156448} related to the Poisson-Dirichlet (0,1) distribution, in which the first attempts to break the stick would fail (with probability $p_1$) and would produce a geometric number of fragments with size 0 at the beginning of the process. For the stickbreaking process the transition probability $\tilde{p}(y,dx)$ is $\frac{1}{y}{\mathbf{1}}_{(0,y]}(x)dx$ for any $y\in[0,1]$ : the stickbreaking process can be seen as the sequence of low records of an i.i.d. sequence $U$ of uniform random variables on $[0,1]$. Of course, whence $\rho^{(n)}$ and $\rho$ are rearranged in decreasing order,  the small initial fragments are rejected at the end or, in the case of $\rho$, they disappear. Thus
\begin{corollary}
\label{pd1} The decreasing rearrangement of $\rho^{(n)}$ converges in law to the Poisson-Dirichlet (0,1) distribution.
\end{corollary}

As regards the second result, for any Lyndon word $\mathtt{w} \in \mathfrak{L}_n$, let $R_{n}(\mathtt{w})$ denotes the length of its standard right
factor, and set $r_n=R_n/n$. We have:

\begin{theorem}
\label{loilimit2} For a totally ordered alphabet with probability
distribution $p$ on its letters, the sequence of normalized lengths $r_n$ of
standard right factors of a random $n$-letters long Lyndon word converges in law, when $n\rightarrow\infty$, to
$$
\mu(dx) = p_{1} \delta_{1}(dx) + (1-p_{1})
\mathbf{1}_{[0,1)}(x)dx,
$$
where $\delta_1$ denotes the Dirac mass on the point $1$ and $dx$
the Lebesgue measure on $\R$. As a consequence the moments of
$r_n$ converge to the corresponding moments of $\mu$.
\end{theorem}

For instance, if $p$ is the uniform distribution on $q$ letters,
then the limit law of the normalized length of the standard right
factor of a random Lyndon word, is
$$
\mu(dx) = \frac1q \delta_{1}(dx) + \frac{q-1}q
\mathbf{1}_{[0,1)}(x)dx.
$$

\subsection{Context}
The Poisson-Dirichlet family of distributions was introduced by Kingman \cite{1975}. This distribution arises as a limit for the size of components of decomposable structures in a variety of settings, as shown by Hansen \cite{MR1293077} or Arratia et al. \cite{MR1702562}. 

When the distribution $p$ is uniform on $q$ letters, i.e.
\[p_k=\frac1q\,1\!\!1_{1\le k\le q},\]
the combinatorics of the Lyndon factorization have connections with that of $q$-shuffles \cite{MR1161056} and of monic polynomials of degree $n$ over the finite field $GF(q)$, as explained in \cite{MR1245159,diaconis}. When  $p$ is  uniform, Corollary \ref{pd1} is well known (cf.  \cite{MR1230136,MR1293077}), through the Golomb correspondance between polynomials and words. Actually, for a uniform $p$, a precise description of the size of Lyndon factors in term of the standard Brownian motion is given in  \cite{MR1264035,MR1230136}. Our contribution is twofold :
\begin{itemize}
\item  Theorem \ref{lyndonlengths} provides a description of the sizes of factors in the Lyndon factorization of random words \textit{depending on their rank} in the factorization. Obviously, the order of factors matters in the Lyndon factorization of words, while it has no meaning in the previously cited papers about polynomials or permutations ;
\item even if we sort the factors' lengths in decreasing order, as in Corollary \ref{pd1}, a proof along the lines of \cite{MR1230136,MR1293077} seems out of reach, for we use a perfectly general distribution $p$ on the alphabet (we only require more than one letter): thus, for a combinatorial proof of our result, a precise description for the distribution of sizes of factors jointly with a count of each letter of the alphabet in each factor would be needed, most likely. For instance, for a general $p$, we were not able to prove, or to disprove, the \textit{conditioning relation}  (cf. \cite[p. 2]{Arratia}) that is usually required for convergence to the Poisson-Dirichlet distribution in such settings.
\end{itemize}
As explained in the next section, we circumvent the combinatorial complexity of the problem with the help of a shuffle trick, Lemma \ref{larcin2}, a multivariate extension of the univariate result that was used in \cite{MR2116634,MR2312436}.  In these two papers (\cite{MR2116634,MR2312436}), an invariance by shuffle is used to analyze the lengths of the 2 factors in the  \emph{standard factorization} of random \emph{Lyndon} words with 2 equiprobable letters (an average case analysis is given in \cite{MR2116634}, and the limit distribution is obtained in \cite{MR2312436}), while our paper uses a multivariate extension of the shuffle trick to obtain the asymptotic Poisson-Dirichlet behaviour in the \emph{Lyndon decomposition} of random words.

Incidentally, because the path has already been cleared by our work on the asymptotic Poisson-Dirichlet behaviour, we provide, in Theorem \ref{loilimit2}, a full generalisation of the result given in  \cite{MR2312436}, that is, the asymptotic distribution of the lengths of the two factors of the standard factorization of random Lyndon words is given here for a general probability distribution on an eventually infinite alphabet.

\section{Sketch of proofs}
\label{sketch}
\subsection{Lyndon factorization and runs}
\label{sketch1}

As seen on Figure \ref{LyndonExample}, the smallest factors in the Lyndon factorization of some word $\mathtt{w}\in \mathcal{A}^{n}$ are usually several monoletters words  $\mathtt{a}_{1}$ at the end of $\mathtt{w}$, but besides that, almost surely, a factor of the Lyndon factorization of a random word  $\mathtt{w}$ begins with \emph{long run} of the letter $\mathtt{a}_1$ :
 the long runs mark the beginnings of the smallest words in the necklace $\langle\mathtt{w}\rangle$, which are also the places where  $\mathtt{w}$  is split into its Lyndon factors, according to the following rule : for instance, in a word containing $9$ long runs with lexicographic ranks going from $1$ to $9$, without ties, the runs could be placed along the word in the following way :
$$
\dots4\dots8\dots3\dots5\dots7\dots1\dots9\dots2\dots6\dots.
$$
In this example,  run 1, both the longest, and the smallest lexicographically, marks the beginning of the first Lyndon factor so that the position of run 1 determines the length of the first Lyndon factor, that englobes runs 9, 2 and 6. The second Lyndon factor begins with run 3 and englobes runs 5 and 7, its length is given by the positions of runs 3 and 1, then Lyndon factors 3 and 4 split at the beginning of run 4, and so on ... Note that this argument breaks down if  some of these runs are tied.
Runs $1$, $3$ and $4$, are \emph{records} of the sequence $\underline{4}8\underline{3}57\underline{1}926$ : with the help of  Lemma \ref{larcin2}, we plan to prove that if $V_{i}$ denote the (normalized) position of the beginning of run $i$ inside the random word $\mathtt{w}$, then the sequence  $V=(V_{i})_{i\ge 1}$ is asymptotically i.i.d. uniform on $[0,1]$, and the Lyndon factors split  $\mathtt{w}$ at positions distributed as the successive minimums (or records) of the sequence $V$. That would be  the stick-breaking construction of the Dirichlet process \cite{MR2615013}. 

However there are catches : to provide the asymptotic behaviour of the complete sequence of Lyndon factors, we need an unlimited supply of long runs, and the lexicographic  ranking of these runs must be unambiguous, without any tie. For the first point, let $H_{n}(\mathtt{w})$ denote the number of runs longer than 
\begin{equation}
\label{longrundef}
r=\left\lceil(1-\varepsilon)\log_{1/{p_1}} n\right\rceil,
\end{equation}
in $\mathtt{w}\in \mathcal{A}^{n}$ (such runs are called \emph{long runs}). According to Lemma \ref{lemRn2},  the sequence $(H_{n})_{n\ge1}$ is an unbounded sequence of random variables if $1>\varepsilon>0$. However, for the second point, the probability that there are several longest runs, tied, and also that there are ties at other positions than the first, is non vanishing. Set 
$$\beta=\max \ac{p_{1},1-p_{1}}, \quad
\tilde{r}=1+\lceil 3\log_{1/\beta} n \rceil.$$
In order to break the ties, each long run of $\mathtt{w}$ has to be appended with a suffix to form a  $\tilde{r}$-letters long factor  of  $\mathtt{w}$, called \emph{long block}. According to Lemma \ref{En2}, but for a vanishing probability, all these factors  of  $\mathtt{w}$ are different, thus they are not tied in the lexicographic order, and they are strictly smaller than the other factors of the same length, since they begin with a long run of $\mathtt{a}_{1}$. Thus their rank in the lexicographic order, with their positions, give the Lyndon factorization of  $\mathtt{w}$, as explained at the beginning of this section.  

Let us give a more formal definition of \emph{long runs} : in
this paper $\varepsilon$ denotes a real number in $(0,1/2)$, and a maximal run of the letter $\mathtt{b}$ in the word $\mathtt{w}$ is a factor $\mathtt{w}_{[k+1,\ell]}$ of $\mathtt{w}$ of the form $\mathtt{b}^{\ell-k}$, such that no factor $\mathtt{w}_{[s,t]}$ with $s\le k+1\le\ell\le t$ is a run of $\mathtt{b}$, unless $s=k+1\le\ell=t$. We usually call $k$ the \emph{position} of the factor $\mathtt{w}_{[k+1,\ell]}$.

\begin{defi}
\emph{(Long runs and short runs)}
\label{longrun1}
We call \emph{long run} (resp.
\textit{short run}) of $\mathtt{w}\in \mathcal A^n$ a maximal run of the letter $\mathtt{a}_1$ with
length at least (resp. smaller than)  $r=\left\lceil(1-\varepsilon)\log_{1/
{p_1}} n\right\rceil$. We denote by $H_{n}(\mathtt{w})$ the number of long
runs of ``$\mathtt{a}_1$'' in $\mathtt{w}$.
\end{defi}

\subsection{A refined factorization}
\label{sketch2}
We introduce 2 refinements of the Lyndon factorization, in smaller factors, the first one according to a simple code. For that, we need the following definition:

\begin{defi}\label{morphism} Set $\mathcal B=\{0,1\}$. Now, let $\varphi$ denote the morphism, from $\mathcal{A}^\star$ to $\mathcal B^\star$, that sends the letter $\mathtt{a}_1$ on the digit $\mathtt{0}$, any other letter of $\mathcal A$ on the digit $\mathtt{1}$, and for any $k$ and any word $\mathtt{w}\in \mathcal A^k$, let $\varphi$ send $\mathtt{w}$  on the word $\varphi(\mathtt{w})=\varphi(\mathtt{w}_{1})\varphi(\mathtt{w}_{2})\dots\varphi(\mathtt{w}_{k})\in  \mathcal B^k$. Let $\varphi_{n}$ denote the restriction of  $\varphi$ to $\mathcal A^n$.
\end{defi}
Let $\mathfrak{S}_{m}$ denotes the set of permutations of $\{1, \dots, m\}$, and let $\tau\in\mathfrak{S}_{m}$. Consider, in $\mathcal{A}^{\star}$, the monoid $\mathcal{M}$ of words that begin, but do not end, with letter $\mathtt{a}_{1}$ ; $\mathcal{M}$ is stable, thus free, and its minimal set of generators 
$$\mathcal{X}=\varphi^{-1}\pa{\ac{\mathtt{0}^{k}\mathtt{1}^{\ell}|\ k,\ell\ge 1}}$$
is  a code, according to \cite[Ch. 1.2]{MR797069}. As a consequence, the factorization in $\mathcal{X}$ of some word   $\mathtt{w}\in\mathcal{M}$  is unique. It follows that the action of $\mathfrak{S}_{m}$ on such a factorization $\mathtt{w}=\mathtt{x}_{1}\mathtt{x}_{2}\dots\mathtt{x}_{m}$  is well defined, if $\tau.\mathtt{w}$ is defined as the word of  $\mathcal{M}$ whose unique factorization is $\mathtt{x}_{\tau(1)}\mathtt{x}_{\tau(2)}\dots\mathtt{x}_{\tau(m)}$. The \emph{orbit} $\mathfrak{O}(\mathtt{w})$ of $\mathtt{w}$ under the action of  $\mathfrak{S}_{m}$ is the set $\ac{\tau.\mathtt{w}\,\vert\, \tau\in\mathfrak{S}_{m}}$.

Thus, if $\tau$ is a \emph{random permutation}, then $\tau.\mathtt{w}$ is uniformly distributed on the orbit $\mathfrak{O}(\mathtt{w})$. But the conditional distribution of $\mathtt{w}$ under $\Pn$, given that $\mathtt{w}$ belongs to some orbit $\mathfrak{O}$ in $\mathcal{M}$, is also the uniform distribution on $\mathfrak{O}$, since  the number of occurences of any letter of $\mathcal{A}$ is the same in $\tau.\mathtt{w}$  and in $\mathtt{w}$, so that $\Pn\pa{\ac{\tau.\mathtt{w}}}=\Pn\pa{\ac{\mathtt{w}}}$.

When $\mathtt{w}\notin\mathcal{M}$, we consider the longest factor $\text{pr}_{\mathcal{M}}\pa{\mathtt{w}}$ of $\mathtt{w}$ that belongs to $\mathcal{M}$ ;  $\text{pr}_{\mathcal{M}}\pa{\mathtt{w}}$ is obtained by erasing eventually a run of non-$\mathtt{a}_{1}$ letters at the beginning of $\mathtt{w}$, and a run of $\mathtt{a}_{1}$    at the end. The lengths of these 2 runs  are $\bigO{1}$ with a  probability close to 1, thus $\text{pr}_{\mathcal{M}}\pa{\mathtt{w}}$ and $\mathtt{w}$ are close in terms of the positions of their long runs, once they are rescaled by a factor $1/\abs{\mathtt{w}}$. A random  permutation of the factors of  $\mathtt{w}$ will then be defined as a permutation of the factors of $\text{pr}_{\mathcal{M}}\pa{\mathtt{w}}$, see the next section.  

However, the probability that the smallest factors of $\mathtt{w}$ in $\mathcal{X}$ are tied in the lexicographic order is  non vanishing when $\abs{\mathtt{w}}$ grows, and that prevents us from reading the Lyndon factorization on each of these factors : we would need to consider sequences of factors to break the ties, and these sequences are not preserved by permutations of factors. The shuffle argument of the next section  would then break. 

To circumvent this problem, consider a new factorization : each factor $\mathtt{x}$ of  $\mathtt{w}$ in $\mathcal{X}$  belongs to some subset $\varphi^{-1}\pa{\mathtt{0}^{k}\mathtt{1}^{\ell}}$ for some $k, \ell\in \N$, so set $\abs{\mathtt{x}}_{\mathtt{0}}=k$ (while $|\mathtt{x}|=k+\ell$) and consider the factorization  $\mathtt{x}_{1}\mathtt{x}_{2}\dots\mathtt{x}_{m}$ of  $\mathtt{w}$ in $\mathcal{X}$. 
In the new factorization, each factor $\mathtt{x}_{i}$ such that $\abs{\mathtt{x}_{i}}_{\mathtt{0}}\ge r$  is merged with its successors $\mathtt{x}_{i+1}$, $\mathtt{x}_{i+2}$, etc ... to form a new factor $\mathtt{y}\in\mathcal{M}$, that we call \emph{long block}, such that $
\abs{\mathtt{y}}\ge \tilde{r}$, but such that $\mathtt{y}$ is minimal in the sense that  $\mathtt{y}$ has no factorization  $\mathtt{uv}$ (in $\mathcal{M}$) with  $\abs{\mathtt{u}}\ge \tilde{r}$. In this way, one obtains a \emph{new factorization} of $\mathtt{w}$, provided that there is enough space between the last long run and the end of the word, or between two long runs, so that the associated long blocks do not overlap. The subset $\tilde{\mathcal{M}}\subset\mathcal{M}$ of words with a such a factorization is again a stable, thus free, monoid, with a new minimal set of generators $\tilde{\mathcal{X}}$, again a code, that includes the long blocks plus the elements $\mathtt{x}\in \mathcal{X}$ such that  $
\abs{\mathtt{x}}_{\mathtt{0}}< r$. In Section \ref{building1}, we  prove that 
$$\lim_{n}\Pn\pa{\text{pr}_{\mathcal{M}}\pa{\mathtt{w}}\in\tilde{\mathcal{M}}}=1,$$
and more.

\subsection{The factorization shuffle}
\label{sketch3}
When $\text{pr}_{\mathcal{M}}\pa{\mathtt{w}}\in\tilde{\mathcal{M}}$, there exists a unique factorization $$\mathtt{x}_{0}\text{pr}_{\mathcal{M}}\pa{\mathtt{w}}\mathtt{x}_{m+1}=\mathtt{x}_{0}\mathtt{y}_{1}\mathtt{y}_{2}\dots\mathtt{y}_{m}\mathtt{x}_{m+1}$$ of $\mathtt{w}$, in which $\mathtt{y}_{i}\in\tilde{\mathcal{X}}$ for $1\le i\le m$, $\mathtt{x}_{0}$ is a run of non-$\mathtt{a}_{1}$ letters at the beginning of  $\mathtt{w}$, and $\mathtt{x}_{m+1}$ is the final run of the letter $\mathtt{a}_{1}$ ; if  $\mathtt{w}\in\tilde{\mathcal{M}}$, $\mathtt{x}_{0}=\mathtt{x}_{m+1}=\emptyset$.
For $\omega\in\mathfrak{S}_{m}$, set $$\omega.\mathtt{w}=\mathtt{x}_{0}\mathtt{y}_{\omega(1)}\mathtt{y}_{\omega(2)}\dots\mathtt{y}_{\omega(m)}\mathtt{x}_{m+1},$$
and let
$$\mathtt{x}_{1}\preceq\mathtt{x}_{2}\preceq\dots\preceq\mathtt{x}_{m}$$
denote the sequence of factors of $\text{pr}_{\mathcal{M}}\pa{\mathtt{w}}$ sorted in increasing lexicographic order : almost surely, the first $\Theta\pa{n^{\varepsilon}}$ terms at the beginning of the sequence form  a \emph{strictly} increasing subsequence, according to Lemmas \ref{lemRn2} and \ref{En2}. The orbit of $\mathtt{w}$ under $\mathfrak{S}_{m}$ is 
$$\mathfrak{O}(\mathtt{w})=\ac{\mathtt{x}_{0}\mathtt{x}_{\omega(1)}\mathtt{x}_{\omega(2)}\dots\mathtt{x}_{\omega(m)}\mathtt{x}_{m+1}\ \vert\ \omega\in\mathfrak{S}_{m}}.$$
This is an extension, to $\text{pr}_{\mathcal{M}}^{-1}\pa{\tilde{\mathcal{M}}}$, of the definition of $\mathfrak{O}(\mathtt{w})$ when $\mathtt{w}\in\tilde{\mathcal{M}}$. More generally, any such orbit $\mathfrak{O}$  is characterized by $m$ and by the sequence $\pa{\mathtt{x}_{i}}_{0\le i\le m+1}$, in which $\pa{\mathtt{x}_{i}}_{1\le i\le m}$ is a sorted sequence of elements of $\tilde{\mathcal{X}}$, $\mathtt{x}_{0}$ is a run of non-$\mathtt{a}_{1}$ letters at the beginning of any element $\mathtt{v}\in\mathfrak{O}$, and $\mathtt{x}_{m+1}$ is the final run of the letter $\mathtt{a}_{1}$. 

Each element $\omega.\mathfrak{O}=\mathtt{x}_{0}\mathtt{x}_{\omega(1)}\mathtt{x}_{\omega(2)}\dots\mathtt{x}_{\omega(m)}\mathtt{x}_{m+1}$ of  $\mathfrak{O}$ yields  a partition of $[0,1)$ into  a (large) number ($m+2$) of subintervals with small widths
$$x_{j}=\dfrac{\abs{\mathtt{x}_{j}}}{\abs{\mathtt{w}}}, \quad 0\le j\le m+1,$$ 
in which the interval $[V_{i}(\omega),V_{i}(\omega)+x_{i})$ filled by the factor $\mathtt{x}_{i}$ depends on $\omega$ through its position : this position $V_{i}(\omega)$ is the sum of lengths of the factors $\mathtt{x}_{j}$ on the left of  $\mathtt{x}_{i}$,
$$V_{i}(\omega)=x_{0}+\sum_{j=1}^{m}x_{j}\ind_{\omega^{-1}(j)<\omega^{-1}(i)},$$
and the factor $\mathtt{x}_{j}$ is on the left of  $\mathtt{x}_{i}$ iff $\omega^{-1}(j)<\omega^{-1}(i)$. When convenient, set $V_{i}(\omega)=1$ for $i>m$. Then the successive records to the left, for the sequence $(V_{i})_{i\ge 1}$,  give the positions of the first factors in the Lyndon decomposition, as explained at Section \ref{sketch1}. Set :
\[\norm{\mathfrak{O}}_{2}=\norm{\pa{x_{j}}_{0\le j\le m+1}}_{2},\quad \norm{\mathfrak{O}}_{\infty}=\max\ac{x_{j},\ 0\le j\le m+1}.\]

As noted previously for the factorization in $\mathcal{X}$, the conditional distribution of $\mathtt{w}$ under $\Pn$, given that $\mathtt{w}$ belongs to some orbit $\mathfrak{O}$ with $m$ factors in $\tilde{\mathcal{X}}$, $\Pn\pa{.|\mathfrak{O}}$, is  uniform, but, for a nonrandom word $\mathtt{w}\in\mathfrak{O}$, if $\omega$ is a \emph{random permutation} in $\mathfrak{S}_{m}$, then $\omega.\mathtt{w}$ is uniformly distributed on the orbit $\mathfrak{O}$ too. 
We prove in this Section that, under 
$\Pn\pa{.|\mathfrak{O}}$, $V_{[d]}=(V_{i})_{1\le i\le d}$ is close to uniform on $[0,1]^{d}$, under mild conditions on $\mathfrak{O}$. More specifically, let $\mathbb U_A$ (resp. $ {\mathbb U}_{d}$) denote  the uniform probability distribution on a finite set $A$ (resp. the uniform distribution on $[0,1]^d$,  for a given integer $d$). Let $U=(U_i)_{i\ge 1}$ denote a sequence of i.i.d. random variables uniform on $[0,1]$, and let $U_{[d]}$ denote the sequence of its $d$ first terms, distributed according to  $ {\mathbb U}_{d}$. Recall that the $L_2$-Wasserstein metric $\mathcal{W}_2(.,.)$ is defined by
\begin{eqnarray}
\label{metric2}
\mathcal{W}_2(\mu,\nu)
&=&
\inf_{{\scriptstyle\mathcal{L}(X)=\mu}\atop{\scriptstyle\mathcal{L}(Y)=\nu}}\esp{}{\left\|X-Y\right\|_2^2}^{1/2},
\end{eqnarray}
in which $\mu$ and $\nu$ are probability distributions on $\mathbb R^d$, and $\norm{.}_2$ denotes the Euclidean norm  on $\mathbb R^d$. Convergence of $\mathcal L(X_n)$ to  $\mathcal L(X)$ with respect to   $\mathcal{W}_2(.,.)$ entails convergence of $X_n$ to $X$ in distribution (see \cite{Rachev}). The \emph{shuffle lemma} asserts that the Wasserstein distance between $V_{[d]}$ and $U_{[d]}$ is bounded by a simple expression depending on the maximal width of the subintervals in $\mathfrak{O}$. 

\begin{lemma}[Shuffle lemma]
\label{larcin2} For $m\ge d\ge1$, and for any orbit $\mathfrak{O}$ with sequence of factors $(\mathtt{x}_{i})_{0\le i\le m+1}$, 
\[\mathcal{W}_2(V_{[d]},\mathbb U_d)\le\sqrt{d/3}\norm{\mathfrak{O}}_{2}\le\sqrt{d\norm{\mathfrak{O}}_{\infty}/3}.\]
\end{lemma}
As a consequence of Propositions \ref{Gn3} and \ref{blocklength}, a.s. under $\Pn$, $\norm{\mathfrak{O}}_{\infty}=\bigO{\ln(n)/n}$ almost surely.
\begin{proof} The proof is similar to the proof of  \cite[Lemma 6.3]{MR2312436}, which is the special case $k=1$ of Lemma  \ref{larcin2}. As in the proof of  \cite[Lemma 6.3]{MR2312436}, we set
\[
\tilde{V}_i=x_0+\sum_{j:\,1\le j\le m,\atop\text{ and }U_j<U_i}\,x_j,
\]
and we note that $\tilde{V}_{[d]}=(\tilde{V}_i)_{1\le i\le d}$ has the same distribution as ${V}_{[d]}$. Among the many couplings between ${V}_{[d]}$ and ${U}_{[d]}$, this special one provides the desired bound on the Wasserstein distance. Actually, for some $j\in [\![1,d]\!]$, conditioning given $U_j$ and using $\sum x_{j}=1$, we obtain
\begin{eqnarray*}
\esp{}{(U_j-\tilde{V}_j)^2} &=&
\esp{}{\pa{x_0(U_j-1)+\sum_{i=1}^m x_i\pa{U_j-1_{\{U_i<U_j\}}}\ +x_{m+1}U_j}^2}
\\
&=&
\esp{}{(1-U_j)^2}\ x_0^2+\esp{}{U_j^2}\ x_{m+1}^2+\esp{}{U_j(1-U_j)}\ \sum_{i=1}^m x_i^2
\\
&=&\tfrac13\ \pa{x_0^2+x_{m+1}^2}\ +\ \tfrac16\ \sum_{j=1}^m\ x_j^2\ \le\ \tfrac13\ \norm{\mathfrak{O}}_{2}^{2}.
\end{eqnarray*}
The result follows from
\[\mathcal{W}_2(V_{[d]},\mathbb U_d)^{2}\le \sum_{j=1}^{d}\esp{}{(U_j-\tilde{V}_j)^2}.\] We struggle with the idea that such computations are new. Actually the argument can be adapted (taking the $x_i$'s in $\{0,1/n\}$) to compute the $L^2$ 
distance $(\sum t(1-t))/n$ between an evaluation $F_{n}(t)$ of the empirical distribution function and $t$, cf. \cite[Ch. 3.1, p.85, display (3)]{shorack2009empirical}.
\end{proof}

The shuffle lemma delivers the expected result, provided that the factorization of $\mathtt{w}$ in $\tilde{\mathcal{X}}$ behaves as follows, under $\Pn$ or under $\Pnt$ :
\begin{enumerate}
  \item as explained before, the distribution of the random word is invariant under a random uniform shuffle of the factors (subintervals) ;
    \item the length of the factors of $\mathtt{w}$ is typically $o(n)$ (actually, $\bigO{\ln n}$), while the Lyndon factors are $\Theta(n)$, so that  $$\norm{\mathfrak{O}}_{\infty}=\bigO{\ln n/n};$$
    \item almost surely under $\Pn$ or under $\Pnt$, $\text{pr}_{\mathcal{M}}\pa{\mathtt{w}}\in\tilde{\mathcal{M}}$.
\end{enumerate}
Section~\ref{run2} is devoted to preliminary results on some statistics on runs. Specially useful is the observation  that the length of the longest run of the letter $\mathtt{a}_1$ is typically of order $\log_{1/p_{1}}n$. In Section~\ref{building1}, we introduce the set $\mathcal{G}_{n}\subset \mathcal{A}^{n}$ of \emph{good words} that satisfy points 2 and 3, and we prove that the set $\mathcal{G}_{n}$ is almost sure.

Once these preliminary tasks are performed, we use the Shuffle Lemma \ref{larcin2}, to prove the main results, Theorem \ref{lyndonlengths} in Section \ref{Th2}, and Theorem \ref{loilimit2} in Section \ref{Th1}.


\section{Statistical properties of runs, under $\Pn$ or under $\Pnt$}
\label{run2}

For $\mathtt{w}\in\mathcal P_n$, let $\pi(\mathtt{w})$ denote the unique Lyndon word
in the necklace of $\mathtt{w}$.  For $s\ge 1$, we set
\[\norm{p}_{s}=\pa{\sum_{i}p_{i}^{s}}^{1/s}.\]
The next Lemma allows  to translate  bounds on $\P_{n}$ into bounds on $\Pnt$:
\begin{lemma}
\label{usarg2} For $A \subset \mathcal A^n$, we have:
 $$\mid \Pnt (A)- \P_{n}({\pi}^{-1} (A)) \mid = \bigO{\norm{p}_{2}^{n}}.$$
\end{lemma}
Note that $\norm{p}_{1}=1$, and that, under the assumption $\ac{0<p_{1}<1}$, $\norm{p}_{s}$ is strictly decreasing in $s$. Among other well known inequalities, we shall make use of  $\norm{p}_{2}\le\sqrt{\max p_{i}}\le\sqrt{\beta}$. For instance, the choice $A=\mathcal A^n$ leads to
\[\mid 1- \P_{n}(\mathcal P_n) \mid = \bigO{\norm{p}_{2}^{n}}= \bigO{\beta^{n/2}}.\]
Due to Lemma \ref{usarg2}, some  properties of statistics, such as the number of runs  and the length of the longest runs, that behave nicely under cyclic permutations, hold true a.s. for random Lyndon words as soon as they hold true a.s. for random words. That is, the next  Lemmas hold true under  $\Pnt$ as well as under $\P_{n}$. Thus they prepare simultaneously the proofs of Theorems \ref{lyndonlengths} and \ref{loilimit2}. The proofs of the results in this Section are tedious, and thus they are postponed to Section \ref{appendix}. 
\begin{defi}\label{defdebase} For any word $\mathtt{w}\in \mathcal A^n$, let  $N_n(\mathtt{w})$ denote  the number of runs in $\varphi(\mathtt{w})$, let $W_1(\mathtt{w}),W_2(\mathtt{w}), \dots, W_{N_n}(\mathtt{w})$ be their lengths, let $N_n^{(\mathtt{e})}(\mathtt{w})$ (resp. $M^{(\mathtt{e})}_n(\mathtt{w})$) denote the number of runs of the letter $\mathtt{e}$ in the word $\mathtt{w}$ (resp. the maximal length of such runs). Thus, for $\mathtt{w}\in\mathcal{A}^\star$, $N_n^{(\mathtt{a}_{1})}(\mathtt{w})=N_n^{(\mathtt{0})}\left(\varphi(\mathtt{w})\right)$.
\end{defi}
  
First, since Theorem \ref{lyndonlengths} deals with a sequence of factors that begin with a long run of $\mathtt{a}_1$, a large number of such runs is required: 
\begin{lemma}[Number of runs of  the letter  $\mathtt{a}_1$]
\label{Nn2}Set $\sigma^{2}=p_{1}(1-p_{1})$. For $(a,b)\in\R^{2}$ such that $a<\sigma^{2}$, we have
\begin{eqnarray*}
\Pn\pa{N_n^{(\mathtt{a}_{1})}< an+b } &=& \bigO{n^{-1}},
\end{eqnarray*}
and
\begin{eqnarray*}
\Pnt\pa{N_n^{(\mathtt{a}_{1})}< an+b   } &=& \bigO{n^{-1}}.
\end{eqnarray*}
\end{lemma}

We also need some information about the length of the
longest runs of ``$\mathtt{a}_1$'' in a word $\mathtt{w} \in \mathcal A^n$ and in its necklace $\langle \mathtt{w}\rangle$, for,
among these long runs, the longest is expected to be the prefix
of the smallest Lyndon factor of $\mathtt{w}$, or the prefix of the unique Lyndon word  in $\langle \mathtt{w}\rangle$. Also, the second longest is expected to be the prefix of the
second smallest Lyndon factor of $\mathtt{w}$ or the prefix of the standard right
factor of the Lyndon word in  $\langle \mathtt{w}\rangle$. Furthermore,
if Theorem \ref{loilimit2} is to be true,
there should exist at least two long runs and, if Theorem \ref{lyndonlengths} is to be true, the number of these long runs should grow indefinitely with $n$, like the number of Lyndon factors of the random word. These points are consequences of Theorems \ref{lyndonlengths} and \ref{loilimit2}, but they are also 
some of the steps of the proofs of these Theorems. They are addressed by the next Lemmas. According to Definition \ref{longrun1}, $H_{n}(\mathtt{w})$ is the number of long runs in the word $\mathtt{w}\in \mathcal{A}^{n}$.
\begin{lemma}[Number of long runs]
\label{lemRn2}
\begin{eqnarray*}
\Pn\pa{H_{n} \geq \alpha  n^\varepsilon} &=&
1-\bigO{n^{-1}},
\end{eqnarray*}
and
\[
\Pnt\pa{H_{n} \geq \alpha  n^\varepsilon}
=1-\bigO{n^{-1}},
\]
in which $\alpha$ is a positive constant smaller than $\sigma^{2}$.
\end{lemma}

\vspace{.4cm} Recall that $M^{(\mathtt{1})}_n$ denote the length of the largest run of non-$\mathtt{a}_{1}$ letters. We have:

\begin{lemma}[Large values of the longest runs]
\label{lemMnGV2} Under $\Pn$ or $\Pnt$, the probabilities of the events 
$\ac{M^{(\mathtt{0})}_n\ge 2\log_{1/{p_1}}{n}}$
 and 
 $\ac{M^{(\mathtt{1})}_n \ge2\log_{1/(1-{p_1})}{n}}$
 are $\bigO{n^{-1}}$.
\end{lemma}

The asymptotic behaviour of the Lyndon (resp. standard) factorization depends on $p$ only through $p_{1}$, and the reason appears in the proofs  of the previous Lemmas, to be found in Section \ref{appendix} : only the lengths and positions of the runs of $\mathtt{a}_{1}$ matter.


\section{Long blocks of words and good words}
\label{building1}


We mentioned in Section \ref{sketch2} the need to break ties between  the long runs: in order to do just that, we  introduced the stable monoid $\tilde{\mathcal{M}}$ and its minimal set of generators, the code $\tilde{\mathcal{X}}=(\tilde{\mathcal{M}}/\ac{\emptyset})^{2}\backslash(\tilde{\mathcal{M}}/\ac{\emptyset})$. 
\begin{defi}
\label{defi:longblock2} The elements of  $\tilde{\mathcal{X}}$ are of two sorts, that we call \emph{long} and \emph{short blocks} :
\begin{itemize}
\item the \emph{short blocks} are the elements  $\mathtt{x}$ of $\mathcal{X}$ such that $\abs{\mathtt{x}}_{\mathtt{0}}<r$ ;
\item the \emph{long blocks} are the elements  $\mathtt{y}$ of $\mathcal{M}$ that satisfy
\begin{itemize}
\item  $\mathtt{y}$ begin with a long run, 
\item $\abs{\mathtt{y}}\ge \tilde{r}$,
\item  $\mathtt{y}$ is minimal in that $\mathtt{y}\notin\ac{\mathtt{uv}\,\vert\ \mathtt{u},\mathtt{v}\in \mathcal{M}\backslash\ac{\emptyset}\text{~and~}\abs{\mathtt{u}}\ge \tilde{r}}$.
\end{itemize}
\end{itemize}
\end{defi}
When $\mathtt{w}$ belongs to the set $\mathcal G_n\subset\mathcal A^n$ of \emph{good words}, defined below,  $$\norm{\mathfrak{O}(\mathtt{w})}_{\infty}=\bigO{\ln n/n}$$ and Lemma \ref{larcin2} provides the desired asymptotically uniform distribution  for $V_{[d]}$, conditionally given $\mathfrak{O}(\mathtt{w})$, for any $d\ge1$.
\begin{defi}
\label{goodword2} A word $\mathtt{w}\in\mathcal A^n$ is a \emph{good word}
if  it satisfies the following conditions:
\begin{itemize}
\item[\textbf{i.}] $\mathtt{w}$ has at least $\lfloor\alpha n^\varepsilon\rfloor$ long blocks ;
\item[\textbf{ii.}] $\text{pr}_{\mathcal{M}}\pa{\mathtt{w}}\in\tilde{\mathcal{M}}$ ;
\item[\textbf{iii.}] if two long blocks have a common factor, its length is not larger than $\tilde{r}-1$,
\item[\textbf{iv.}] $M^{(\mathtt{0})}_n(\varphi(\mathtt{w})) \leq 2 \log_{1/p_1}{n}$,
 and $M^{(\mathtt{1})}_n(\varphi(\mathtt{w})) \leq 2 \log_{1/(1- p_1)}{n}$.
\end{itemize}
\end{defi}
In this section, we prove  that $\mathcal G_n$  has a large probability: 
\begin{prop}
\label{Gn3} Under $\Pn$  or $\Pnt$, the probability of  $\mathcal G_n^c$ is $\bigO{n^{2\varepsilon-1}\,\ln^2 n}$.
\end{prop}
Recall that $\tilde{r}=1+\lceil 3\log_{1/\beta} n \rceil $. For the proof of Proposition  \ref{Gn3}, we need a few lemmas:
\begin{lemma}
\label{En2}
 Denote by $\mathcal{E}_n$ the set of words   $\mathtt{w}\in\mathcal A^n$  in which  some $\tilde{r}$-letters long factor appears twice in the necklace
  $\langle \mathtt{w}\rangle$, at two non-overlapping positions:
 \[\mathcal{E}_{n}=\ac{\mathtt{w}\in\mathcal A^n\,|\,\exists(\mathtt{w}^{\prime},v,a,b)\in\langle \mathtt{w}\rangle\times\mathcal A^{\tilde{r}}\times(\mathcal A^\star)^2\text{ s.t. }\mathtt{w}^{\prime}=vavb}.\]
Then, under $\Pn$  or $\Pnt$, the probability of  $\mathcal{E}_n$ is $\bigO{n^{-1}}$.
\end{lemma}
A key argument of the proof of the main results breaks down if some long block of the factorization of a random word is a prefix of another long block, somewhere else in the word. In order to preclude that,  we shall consider blocks with at least $\tilde{r}$  letters (at least thrice the length of the longest run(s)\footnote{The probability that there exists several runs with the same maximal length inside a $n$-letters long random word is non vanishing with $n$ large, so $\log_{1/p_1}n$ characters would be too short.} of the letter $\mathtt{a}_1$), and we shall use   Lemma \ref{En2}.
\begin{proof}
We have
\begin{equation}\label{ineqbloc} \Pn(\mathcal{E}_n)  =\bigO{n^2 \
\beta^{\tilde{r}}}= \bigO{n^{-1}},\end{equation} in
which $n^2$ is a bound for the number of positions of the pair of
factors of $\mathtt{w}$, and $\beta^{\tilde{r}}$ is a bound for
the conditional  probability that the second factor is equal to the first
factor, given the value of
the first factor and the positions of the factors.
Due to Lemma \ref{usarg2},
 $\Pnt(\mathcal{E}_n)$ satisfies
\[
\mid\Pnt(\mathcal{E}_n)-\Pn\pa{\pi^{-1}(\mathcal{E}_n)}\mid= \bigO{\beta^{n/2}},
\]
and $\pi^{-1}(\mathcal{E}_n)=\mathcal{E}_{n}\cap\mathcal P_{n}\subset \mathcal{E}_{n}$.
\end{proof}
As mentioned at the very beginning of the paper, a word $\mathtt{v}$ is a factor of $\langle \mathtt{w}\rangle$ as soon as it is a factor of some element $\mathtt{w}^{\prime}\in\langle \mathtt{w}\rangle\subset\mathcal{A}^{n}$.  As a consequence, an $\ell$-letters long factor  $\mathtt{v}$ of  $\mathtt{w}$ can be found at $n-\ell+1$ positions, while such a factor   $\mathtt{v}$ of $\langle \mathtt{w}\rangle$ has $n$ possible positions.
\begin{lemma}[Overlap of long blocks]
\label{Fn2} Let $\mathcal{F}_n$ denote the set of words $\mathtt{w}\in\mathcal A^{n}$
such that some factor  of $\langle \mathtt{w}\rangle$, 
$\lceil7\log_{1/\beta} n\rceil$-letters long, contains two
disjoint long runs.  
 Then, under $\Pn$  or $\Pnt$, the probability of  $\mathcal{F}_n$ is $\bigO{n^{2\varepsilon-1}\,\ln^2 n}.$
\end{lemma}
\begin{proof} The bound for $\Pn(\mathcal{F}_n)$ has a factor $n$ for the position of the $\lceil7\log_{1/\beta} n\rceil$-letters long factor, a factor 
$49 (\log_{1/\beta} n)^2$ (a crude bound) for the positions of the 2 runs inside this factor, and a factor  $n^{2\varepsilon-2}\ge p_{1}^{r}\times p_{1}^{r}$ for the probability of 2 disjoint $r$-letters long runs at 2 specified positions. The proof extends to 
$\Pnt$  according to Lemma \ref{usarg2}.\end{proof}

\begin{lemma}\label{In} Let $\mathcal{I}_n$ denote the set of words
$\mathtt{w}\in\mathcal A^{n}$ whose suffix of length
$2\tilde{r}$ contains a long run of ``$\mathtt{a}_1$''.
Then, under $\Pn$  or $\Pnt$, the probability of  $\mathcal{I}_n$ is $\bigO{n^{2\varepsilon-1}\,\ln^2 n}$.
\end{lemma}

\begin{proof} We have
\[
\Pn(\mathcal{I}_n)
\le
2 n^{\varepsilon-1}\ \tilde{r}.
\]
The factor $2\tilde{r}$ bounds the
number of positions where such a long run could begin. The factor $n^{\varepsilon-1}={p_1}^{(1-\varepsilon) \log_{1/p_1}n}$ is the
probability that a long run begins at some given position. The
result for $\Pnt$ follows from Lemma \ref{usarg2}, Lemma \ref{Fn2}  and
$$\Pn\pa{\pi^{-1}(\mathcal{I}_n)} \le \Pn(\mathcal{F}_n).$$
\end{proof}

\begin{proof}[Proof of Proposition \ref{Gn3}]  Consider the sets
\[
\mathcal{V}_{n}=\ac{\mathtt{w}\in\mathcal{A}^n\,\left|\,
\mathtt{w}\text{ satisfies iv. and }H_{n}(\mathtt{w}) \geq \alpha  n^\varepsilon\right.}
\] 
and 
\[
\tilde{\mathcal{G}}_{n}=\mathcal{V}_{n}\backslash\pa{\mathcal{E}_{n}\cup \mathcal{F}_{n}\cup \mathcal{I}_{n}}.
\]
Then,  under $\Pn$  or $\Pnt$, the probability of  $\tilde{\mathcal{G}}_{n}^c$ is $\bigO{n^{2\varepsilon-1}\,\log^2 n}$, due to Lemmas \ref{lemRn2}, \ref{lemMnGV2}, \ref{En2} and \ref{In}, since, for instance, 
\[\Pn\pa{\tilde{\mathcal{G}}_{n}^c}\le \Pn\pa{\mathcal{V}_{n}^c}+\Pn\pa{\mathcal{E}_{n}}+\Pn\pa{\mathcal{F}_{n}}+\Pn\pa{\mathcal{I}_{n}} .\]
Let us prove that $\tilde{\mathcal{G}}_{n}\subset \mathcal{G}_{n}$ : consider a word  $\mathtt{w}\in \tilde{\mathcal{G}}_{n}$, and in order to prove that $\mathtt{w}$ satisfies conditions  \textbf{i.}  and \textbf{ii.} in Definition \ref{goodword2}, consider a  long run $\mathtt{u}=\mathtt{a}_{1}^{k}$ of $\mathtt{w}=\mathtt{tus}$, with $k\in[\![ r,2 \log_{1/p_1}{n}]\!] $ : is it the prefix of a long block $\mathtt{uv}$ ? 

The  eventual long block beginning with $\mathtt{u}$ ends with the run of non-$\mathtt{a}_{1}$ letters containing the character  $\mathtt{w}_{\abs{\mathtt{t}}+\tilde{r}}$ of the word $\mathtt{w}$, if $\mathtt{w}_{\abs{\mathtt{t}}+\tilde{r}}\neq \mathtt{a}_{1}$, or with the next run of non-$\mathtt{a}_{1}$ letters, if $\mathtt{w}_{\abs{\mathtt{t}}+\tilde{r}}= \mathtt{a}_{1}$. In this last case, $\mathtt{w}_{\abs{\mathtt{t}}+\tilde{r}}$ is part of a short run, else $\mathtt{w}$ would belong to  $\mathcal{F}_{n}$. Thus, in any case,  $$\abs{\mathtt{uv}}\le\tilde{r}+r +2 \log_{1/(1- p_1)}n<2 \tilde{r}.$$ 

Since $\mathtt{w}\notin \mathcal{I}_{n}$, we have 
$$\abs{\mathtt{t}}+\abs{\mathtt{uv}}\le n,$$
thus there is room enough to build a long block beginning with $\mathtt{v}$, by merging $\mathcal{X}$-factors of $\mathtt{w}$. Also $\mathtt{w}\notin\mathcal{F}_{n}$, thus $\mathtt{v}$ does not contain a second long run, and  long blocks overlap does not happen. Finally, $\mathtt{w}$ satisfies \textbf{ii.}, so there exists a long block for each long run, and since $\mathtt{w}\in \mathcal{V}_{n}$, $\mathtt{w}$ satisfies \textbf{i.}. Condition \textbf{iv.} is satisfied by definition of  $\mathcal{V}_{n}$ and \textbf{iii.} is satisfied because  $\mathtt{w}\notin \mathcal{E}_{n}$.
\end{proof}
\begin{prop}
\label{blocklength}
If $\mathtt{w}\in \mathcal G_n$, then $\norm{\mathfrak{O}(\mathtt{w})}_{\infty}\le 2\tilde{r}/n $.
\end{prop}
\begin{proof}
Note that by Definition \ref{defi:longblock2}, the short blocks of some word $\mathtt{w}\in \mathcal G_n$
begin with a short run of less than $r$ letters $\mathtt{a}_1$, and due to point \textbf{iv.} of Definition \ref{goodword2},   a run of less than   $2 \log_{1/(1- p_1)}{n}$ ``letters'' $\bar {\mathtt{a}}_1$, thus short blocks have less than 
$$r+2 \log_{1/(1- p_1)}{n}\le \tilde{r}$$
 letters, while its long blocks are not longer than
$$\tilde{r} -1+r-1+\lfloor 2\log_{1/(1-
p_1)} n\rfloor\le 2\tilde{r}.$$ 
For a long block, count $\tilde{r}$ letters for the minimal size
of a long block, plus eventually a run of ``$\mathtt{a}_1$'' (a short one,
due to point \textbf{ii.} of Definition \ref{goodword2}, at most
$r-1 $-letters long
starting before the $\tilde{r} $-limit)
and  a run of ``$\bar a_1$'', at most $\lfloor 2\log_{1/(1-
p_1)} n\rfloor$ letters, due to point
\textbf{iv.} of Definition \ref{goodword2}.
\end{proof}

In the two following sections, we prove separately the
main theorems, Theorem \ref{lyndonlengths} and Theorem
\ref{loilimit2}.

\section{Proof of Theorem \ref{lyndonlengths}}
\label{Th2}
Set $$s_{n,i}(\mathtt{w})=1-(\rho_{n,1}(\mathtt{w})+\rho_{n,2}(\mathtt{w})+\dots+\rho_{n,i}(\mathtt{w})) ;$$ $s_{n,i}(\mathtt{w})$ is the normalized position of the $i$th factor of the Lyndon decomposition of $\mathtt{w}$, meaning that the $i$th factor of $\mathtt{w}$ is $\mathtt{w}_{[n\,s_{n,i},n\,s_{n,i-1}-1]}$. The correspondance between 
$s^{(n)}=(s_{n,i})_{i\ge1}$ and $\rho^{(n)}$ is bicontinuous on $[0,1]^{\mathbb N}$, thus Theorem \ref{lyndonlengths} is equivalent to the convergence in distribution of $s^{(n)}$ to $s$, $s$ being the variant  of the stickbreaking process defined at Section \ref{mainresults}.

For any word $\mathtt{w}\in \mathcal G_n$, according to condition \textbf{ii.} of Definition \ref{goodword2}, $\text{pr}_{\mathcal{M}}\pa{\mathtt{w}}$ belongs to $\tilde{\mathcal{M}}$ and, as such, it has a unique factorization in $\tilde{\mathcal{X}}$, thus we can write :
\begin{equation}
\label{faktorn}
\mathtt{w}
=
\overline{\mathtt{a}}_1^{k(\mathtt{w})}Y_1(\mathtt{w})\dots Y_{M_n(\mathtt{w})-1}(\mathtt{w})Y_{M_n(\mathtt{w})}(\mathtt{w})\mathtt{a}_1^{L_{n}(\mathtt{w})},
\end{equation}
in which $k\geq 0$, $L_{n}(\mathtt{w})\geq 0$ and the $Y_i$'s are elements of  $\tilde{\mathcal{X}}$, either long blocks, or short blocks. Let $J_{i,n}(\mathtt{w})$, $1\le i\le H_n(\mathtt{w})$, denote the index of the $i$-th smallest block of $\mathtt{w}\in \mathcal G_n$: since $i\le H_n(\mathtt{w})$, $Y_{J_{i,n}}$ has to be a long block, and there are no ties among long blocks, so that $J_{i,n}$ is well defined on  $\mathcal G_n$. Let $V_{i,n}(\mathtt{w})$, $1\le i\le H_n(\mathtt{w})$,  denote the normalized position of $Y_{J_{i,n}}(\mathtt{w})$ in $\mathtt{w}$, defined as the ratio $|\mathtt{u}|/|\mathtt{w}|$, in which $\mathtt{w}$ has the factorization $\mathtt{w}=\mathtt{u}Y_{J_{i,n}}(\mathtt{w})\mathtt{v}$. The normalized position $V_{i,n}(\mathtt{w})$ is given by the formula :
\begin{equation}
\label{defbegin}
V_{i,n}(\mathtt{w})
=
\frac1{|\mathtt{w}|} \pa{k(\mathtt{w})+\sum_{j=1}^{J_{i,n}(\mathtt{w})-1} |Y_j(\mathtt{w})|}\ ;\hspace{.5cm}i=1,\dots,H_n(\mathtt{w}).
\end{equation}
For a word $\mathtt{w}\in \mathcal G_n$, it is convenient to complete the sequence $(V_{i,n}(\mathtt{w}))_{1\le i\le H_n(\mathtt{w})}$ by an infinite sequence of 0's, in order to form an infinite sequence 
$$V^{(n)}(\mathtt{w})=(V_{i,n}(\mathtt{w}))_{i\ge 1},$$
and for a word $\mathtt{w}\in\mathcal{A}^n\backslash\mathcal{G}_n$, let $V^{(n)}(\mathtt{w})$ be an infinite sequence of 0's.
For a word $\mathtt{w}\in\mathcal{G}_n$, this is not much of a perturbation, since the original sequence is very long : according to Lemma \ref{lemRn2},  the probability that $H_n(\mathtt{w})$ is smaller than $\alpha n^\varepsilon$ vanishes. 

Now let us address the $L_{n}$ first terms of $s^{(n)}$ : they form the sequence
$$\tfrac1n\ (n-1,n-2,\dots,n-L_{n}-1,n-L_{n}),$$
If $L_{n}(\mathtt{w})= \ell\ge 1$, the first $\ell$ factors of the Lyndon factorization are $\ell$ words reduced to one letter ``$a_{1}$''. Thus, for $1\le i\le \ell$,
\[s_{i,n}=1-\frac in.\]
Let $L$ denote a geometric random variable with parameter $1-p_{1}$, such that, for $\ell\ge 0$,
\[\mathbb P(L=\ell)= p_{1}^\ell(1-p_{1}),\]
and let $\xi$ denote its probability distribution. It turns out that $L_{n}$ converges in distribution to $L$. Let  $\Lambda_{n}$  denote the number of low records of the  sequence $(V_{i,n}(\mathtt{w}))_{1\le i\le H_n(\mathtt{w})}$. Then, as explained at Section \ref{sketch1}, the next $\Lambda_{n}$ terms of the sequence $s^{(n)}$, i.e. $s_{k+1,n}, s_{k+2,n}$, $\dots, s_{k+\Lambda_{n},n}$, are the low records of the sequence $(V_{i,n}(\mathtt{w}))_{1\le i\le H_n(\mathtt{w})}$. In light of this, the proof has $4$ steps :
\begin{enumerate}
\item $s$ can be described as the sequence of low records of a sequence $U=(U_i)_{i\ge 1}$   of i.i.d. random variables uniform on $[0,1]$,  appended with a prefix sequence of $L$ 1's, $L$ and $U$ independent, an operation that we denote $\mathfrak{A}(L,U)$, and that we  define more formally below ;
\item $(L_{n},V^{(n)})$ converges in distribution to $(L,U)$ ;
\item provided that $\mathfrak{A}$ has some regularity properties, $\mathfrak{A}(L_{n},V^{(n)})$   converges in distribution to $\mathfrak{A}(L,U)=s$ ;
\item since  $\lim_{n}\Lambda_{n}=+\infty$ in some sense, cf. Proposition \ref{nombrecord} below, $s^{(n)}$ and   $\mathfrak{A}(L_{n},V^{(n)})$ are close in some sense.
\end{enumerate}
The following bound is proven at Section \ref{secnombrecord} :
\begin{prop} For $\varepsilon$ previously chosen in $(0,1)$,
\label{nombrecord}
$$
\mathbb P_n(\Lambda_{n}\le \varepsilon\log n/3) = \bigO{\frac 1{\log n}}.
$$
\end{prop}

For points 1 and 3, let $T$ be the functional that shifts a sequence $u$ as follows:
\[T(u)=T(u_{1},u_{2},\dots)=(1,u_{1},u_{2}, \dots).\] 
Let $S$ be the functional that keeps track of the sequence of low records (in the broad sense) of a sequence $u$ of real numbers. 
The functional $S$ is well defined and is continuous on a set of measure 1 of $[0,1]^{\mathbb N}$, for instance on the set $\mathcal R$ of sequences $u$ without repetition such that $\liminf u=0$. Then the functional $\mathfrak{A}$ defined on $\mathbb N \times\mathcal R$ by
\[\mathfrak{A}(k,u)= T^k\circ S(u)\]
is continuous as well. Now $s$ has the same distribution as $\mathfrak{A}(L,U)$ as a 
consequence of the Markov property of $s$ and of the particular form of its transition kernel.

For point 4, note that the difference between the two sequences $s^{(n)}$ and $ \mathfrak{A}(L_{n},V^{(n)})$ is 
\[
\mathfrak{A}(L_{n},V^{(n)})-s^{(n)}
=
\pa{\tfrac1n,\tfrac2n,\dots,\tfrac kn,0,0,\dots,0,s_{k+\Lambda_{n}+1,n},s_{k+\Lambda_{n}+2,n},\dots}.
\]
Endowing $[0,1]^{\mathbb N}$ with the distance
\[
d(u,v)
=
\sum_{k\ge 1}2^{-k}|u_{k}-v_{k}|,
\]
we obtain
\[
d(s^{(n)},\mathfrak{A}(L_{n},V^{(n)}))
\le
\frac{8}{n}+2^{-L_{n}-\Lambda_{n}}\le
\frac{8}{n}+2^{-1-\Lambda_{n}}.
\]
This inequality and Proposition \ref{nombrecord} entail that 
\[
\esp{}{d(s^{(n)},\mathfrak{A}(L_{n},V^{(n)}))}
\le
\frac{8}{n}+\bigO{\frac 1{\log n}}+n^{-\varepsilon\ln(2)/3}.
\]
According to \cite{Billing2}[Th. 3.1], the previous bound entails that if $\mathfrak{A}(L_{n},V^{(n)})$ converges in distribution to $\mathfrak{A}(L,U)$,  then $s^{(n)}$ converges in distribution  to $\mathfrak{A}(L,U)$ too. Finally, for point 2, note that, under $\mathbb P_{n}$, $L_{n}$ has the same distribution $\xi_{n}$ as $L\wedge n$. This is perhaps clearer when one considers the word  $\overline{\mathtt{w}}$ obtained by reading the word $\mathtt{w}$ from right to left:  clearly, under $\mathbb P_{n}$, $\overline{L}_{n}$ defined by
\[\overline{L}_{n}(\mathtt{w})=L_{n}(\overline{\mathtt{w}})\]
has the same distribution as $L_n$, for $\mathtt{w}$ and $\overline{\mathtt{w}}$ have the same weight. But $\overline{L}_{n}$ has the same distribution as $L\wedge n$. Thus $L^{2}$ convergence of $L\wedge n$ to $L$ entails that
\begin{equation}
\label{geomwasser}
\mathcal{W}_2\pa{\xi_{n},\xi}
\pa{=\mathcal{W}_2\pa{\xi_{n}\otimes\mathbb U_d,\xi\otimes\mathbb U_d}}
= \bigO{n^{2}e^{-n}}.
\end{equation}

Let $\Pnb$ denote the conditional probability given $\mathcal G_n$:
\[\Pnb\pa{A}=\frac{\Pn\pa{A\cap\mathcal G_n}}{\Pn\pa{\mathcal G_n}}.\]
\begin{theorem}Under $\Pn$ or under $\Pnb$, $V^{(n)}$ converge in distribution to $U$.
\end{theorem}
\begin{proof}Set $$V_{[d]}^{(n)}=\pa{V_{i,n}}_{1\le i\le d},$$ and let $\nu_{{d,n}}$ be the the distribution of $V_{[d]}^{(n)}$  under $(\mathcal G_n,\Pnb)$. Due to \cite[Theorem 3.29]{Kallenberg}, we need to prove that, for any $d\ge 1$, the sequence  $V_{[d]}^{(n)}$ is, under $\mathbb G_n$, asymptotically uniform on $[0,1]^d$, which results from considerations in Sections \ref{sketch2} and \ref{sketch3} :
\begin{lemma}[Positions of the $d$ first smallest blocks]
\label{distance3}
We have
\[
\mathcal{W}_2\pa{\nu_{d,n},{\mathbb U_{d}}}
\le\sqrt{2d\tilde{r}/3n}.
\]
\end{lemma}

\begin{proof} For $\mathtt{w}\in\mathcal G_n$, let $\tau\in\mathfrak{S}_{M_n(\mathtt{w})}$ act on $\mathtt{w}$ by permutation of blocks :
$$
\tau.\mathtt{w}
=
\overline{\mathtt{a}}_1^{k}Y_{\tau(1)}(\mathtt{w})\dots Y_{\tau(M_n(\mathtt{w}))}(\mathtt{w})\mathtt{a}_1^{L_{n}(\mathtt{w})}.
$$
As in Section \ref{sketch}, let $\mathfrak{O}(\mathtt{w})$ denote its orbit under that action.
Let   $\mathfrak C_n$ denote the $\sigma$-algebra generated by $\mathcal C_n=\ac{\mathfrak{O}(\mathtt{w})\,;\ \mathtt{w}\in \mathcal G_n}$.  For $\mathfrak{O}\in\mathcal C_n$, let $\nu_{\mathfrak O,n}$ denote the conditional distribution of $V_{[d]}^{(n)}(\mathtt{w})$ given that $\mathtt{w}\in \mathfrak O$. Let $X(\mathtt{w})=\pa{X_i(\mathtt{w})}_{i\ge 0}$ denote the sequence of long blocks of $\text{pr}_{\mathcal{M}}\pa{\mathtt{w}}$ sorted in increasing
lexicographic order, ended by an infinite sequence of empty words, i.e. $X_i(\mathtt{w})=Y_{J_{i,n}}(\mathtt{w})$, if $1\le i\le H_n(\mathtt{w})$, else $X_i(\mathtt{w})=\emptyset$. 
Let $\Xi(\mathtt{w})=\pa{\Xi_i(\mathtt{w})}_{i\ge 0}$ be the corresponding
sequence of lengths. We have :
\begin{lemma}
\label{combine} The weight $p(.)$, $X$, $\Xi$, $H_n$, $L_{n}$ and $M_n$ are
$\mathfrak C_n$-measurable, and
$$
\Pnb = \sum_{\mathfrak{O}\in\mathcal C_n}\frac{\mathbb P_n(\mathfrak{O})}{\mathbb P_n(\mathcal G_n)}\ \mathbb U_\mathfrak{O}.
$$
Also, $\nu_{\mathfrak O(\mathtt{w}),n}$ is  the image of the uniform
probability on $\mathfrak{S}_{K_n(\mathtt{w})}$ by the application $\tau
\longmapsto  V_{d}^{(n)}(\tau.\mathtt{w})$.\end{lemma}
\begin{proof} The weight $p(\mathtt{w})$ depends only on the number of
letters $a_{1}$, $a_{2}$, \dots that compose the word $\mathtt{w}$, not on the
order of these letters in $\mathtt{w}$, thus $p(.)$ is constant on each
$\mathfrak{O}\in\mathcal C_n$: as a consequence, under $\Pnb$, the conditional
distribution of $\mathtt{w}$ given that $\mathtt{w}\in \mathfrak{O}$ is  $\mathbb U_\mathfrak{O}$. The statement about $X$, $\Xi$, $H_n$, $L_n$ and $M_n$ holds true because $\tilde{\mathcal{X}}$ is a code. Since $\mathcal C_n$ is a
partition of $\mathcal G_n$,  the relation in
Lemma \ref{combine} is  the desintegration of $\Pnb$
according to its conditional distributions given $\mathcal C_n$. The last part holds true because the distribution of $\tau
\longmapsto  \tau.\mathtt{w}$ is $\mathbb U_{\mathfrak O(\mathtt{w})}$.
\end{proof}

According to Lemma \ref{larcin2} and \eqref{defbegin},
\begin{equation}
\label{larcin3}
\mathcal{W}_2\pa{\nu_{\mathfrak{O},n},\mathbb U_d} \le \sqrt{d\norm{\mathfrak{O}}_{\infty}/3}.
\end{equation}
Then, Proposition \ref{blocklength}, with the desintegration in Lemma \ref{combine},  give the desired result. \end{proof}
This is for the proof under $\mathbb G_n$. Note that the conditional law $\tilde \nu$, given $A$, of a  $[0,1]^{d}$-valued random variable $X$, defined on a probabilistic space $\Omega$, is Wasserstein-close to its unconditional law $\nu$, if $A$ is close to $\Omega$ :
\begin{equation}
\label{conditionequn}
\mathcal{W}_2\pa{\nu,\tilde \nu}
\le
\sqrt{d\ \pr{\Omega\backslash A}}.
\end{equation}
As a consequence, Lemma \ref{distance3}, together with Proposition \ref{Gn3}, entails that, under $\Pn$,
\[
\mathcal{W}_2\pa{V_{[d]}^{(n)},{\mathbb U_{d}}}
 = \bigO{n^{-1/2+\varepsilon}\
\log n},
\]
for any $d\ge 1$, which ensures the convergence of $V^{(n)}$ to $U$ under  $\Pn$ too.
\end{proof}
The proofs of the weak convergence of $L_{n}$ and $V^{(n)}$, respectively, are complete, now the weak convergence of $(L_{n},V^{(n)})$ is a consequence of the asymptotic independence between  $L_{n}$ and $V^{(n)}$ : all the conditional distributions $\nu_{\mathfrak{O},n}$ have the same limit $\mathbb U_d$, thus $V_{[d]}^{(n)}$ is asymptotically independent of $\mathfrak C_n$, while, according  to Lemma \ref{combine}, $L_{n}$ is  $\mathfrak C_n$ measurable, i.e. constant on each $\mathfrak{O}$ (equal to $L_{n}(\mathfrak{O})$). More precisely, step by step,
\begin{itemize}
\item the distribution $\chi_{n}$ of $(L_{n},V_{[d]}^{(n)})$, under $\mathbb G_{n}$, has the desintegration 
$$
\chi_{n} = \sum_{\mathfrak{O}\in\mathcal C_n}\frac{\mathbb P_n(\mathfrak{O})}{\mathbb P_n(\mathcal G_n)}\ \delta_{L_{n}(\mathfrak{O})}\otimes\nu_{\mathfrak{O},n}\ ;
$$
\item relation \eqref{larcin3} has the straightforward extension \[
\mathcal{W}_2\pa{\delta_{L_{n}(\mathfrak{O})}\otimes\nu_{\mathfrak{O},n},\delta_{L_{n}(\mathfrak{O})}\otimes\mathbb U_d} \le \sqrt{d\norm{\mathfrak{O}}_{\infty}/3}\ ;
\]
\item due the desintegration of $\chi_{n}$, the previous bound entails that, under $\mathbb G_{n}$,\[
\mathcal{W}_2\pa{\chi_{n},{\xi_{n}\otimes\mathbb U_{d}}}
\le\sqrt{2d\tilde{r}/3n}.
\]
\item under $\Pn$,  $\mathcal{W}_2\pa{\chi_{n},{\xi_{n}\otimes\mathbb U_{d}}}$ still vanishes due to \eqref{conditionequn}, and \eqref{geomwasser} completes the proof of point 2.
\end{itemize}

Note that the largest (and shortest) Lyndon factors, that  begin with short blocks, or that do not begin with letter $\mathtt{a}_{1}$, do not appear in this list of $H_n$ factors, but, as a consequence of Theorem \ref{lyndonlengths},
 the total length of these largest factors is $o(n)$ : once normalized by $n$, their lengths do not contribute to the asymptotic behavior of the factorization.




\section{Proof of Theorem \ref{loilimit2}}
\label{Th1}




In this section, we extend  \cite[Theorem 6.4]{MR2312436}   to a general distribution on an infinite alphabet. The proof is similar to that of Theorem \ref{lyndonlengths} or of \cite[Theorem 6.4]{MR2312436}. For $n\ge2$, $\mathtt{w}\in\mathfrak{L}_{n}$ entails that $\text{pr}_{\mathcal{M}}\pa{\mathtt{w}}=\mathtt{w}$, thus $\mathfrak{L}_{n}\subset\mathcal{M}$,  and, according to Definition \ref{goodword2}, $\mathcal G_n\cap\mathfrak{L}_n\subset\tilde{\mathcal{M}}$.  As a consequence, any $\mathtt{w}\in \mathcal G_n\cap\mathfrak{L}_n$   has a unique factorization according to the code $\tilde{\mathcal{X}}$:
\[
\mathtt{w} = Y_0(\mathtt{w})Y_1(\mathtt{w})\dots Y_{K_n(\mathtt{w})-1}(\mathtt{w})Y_{K_n(\mathtt{w})}(\mathtt{w}),
\]
in which the $Y_i$'s stand either for a long block or for a short block. Moreover, $Y_0(\mathtt{w})$ is the smallest block. 

Let $\pnh$ denote the conditional distribution given that $\mathtt{w}\in \mathcal G_n\cap\mathfrak{L}_n$. As in the previous section, let $V_{2,n}(\mathtt{w})$ denote the normalized position of the second smallest block 
in the factorization of $\mathtt{w}$ according to the code $\tilde{\mathcal{X}}$, and let $\nu_n$ 
denotes the distribution of $V_{2,n}$ under $(\mathcal G_n\cap\mathfrak{L}_n,\pnh)$. We first prove that:
\begin{theorem}For the distribution of the position of the second smallest block, it holds that:
\label{distance2}
\[
\mathcal{W}_2\pa{\nu_n,\mathbb U_1} =
\bigO{\sqrt{\frac{\log n}{n}}}.
\]
As a consequence, under  $\pnh$, the moments of $V_{2,n}$ converge to
the corresponding moments of $\mathbb U_1$.
\end{theorem}

\begin{proof} The proof of Theorem \ref{lyndonlengths} holds step by
step: for $\mathtt{w}\in \mathcal G_n\cap\mathfrak{L}_n$ and $\tau\in\mathfrak{S}_{K_n(\mathtt{w})}$, set
\[
\tau.\mathtt{w} = Y_0(\mathtt{w})Y_{\tau(1)}(\mathtt{w})Y_{\tau(2)}(\mathtt{w})\dots Y_{\tau(K_n(\mathtt{w})-1)}(\mathtt{w})Y_{\tau(K_n(\mathtt{w}))}(\mathtt{w}),
\]
and note that $\tau.\mathtt{w}$ still belongs to $\mathcal G_n\cap\mathfrak{L}_n$. Let $\mathfrak{O}(\mathtt{w})$ denote the orbit of $\mathtt{w}$ and let $\mathcal C_{n}$ be the class of orbits of elements of $\mathcal G_n\cap\mathfrak{L}_n$.  If $\mathfrak{O}\in\mathcal C_{n}$ and if $\nu_\mathfrak{O}$ is the conditional distribution of $V_{2,n}(\mathtt{w})$
given that $\mathtt{w}\in \mathfrak{O}$, then $\nu_\mathfrak{O}$ is also the image of the uniform
probability on $\mathfrak{S}_{K_n(\mathtt{w})}$ by the application $\tau
\longmapsto  V_{2,n}(\tau.\mathtt{w})$. Thus Lemma \ref{larcin2} leads to
\[
\mathcal{W}_2\pa{\nu_\mathfrak{O},\mathbb U_1}  \le\sqrt{\norm{\mathfrak{O}}_{\infty}/3}.
\]
Then, Proposition \ref{blocklength}, with the desintegration of $\nu_{n}$ along $\mathcal C_{n}$, entails
\[
\mathcal{W}_2\pa{\nu_n,\mathbb U_1} \le 
\sqrt{2\tilde{r}/n}.
\]
\end{proof}

Now, let us draw some additional consequences, for good \emph{Lyndon} words, of Definition \ref{goodword2}. 

\begin{prop}
\label{consGn2} A good Lyndon word $\mathtt{w}\in\mathcal G_n\cap\mathcal
L_n$ satisfies the following points:
\begin{enumerate}
  \item \label{point:overlap} each long block, by definition a factor of $\langle \mathtt{w}\rangle$, is also a factor of $\mathtt{w}$ ;
    \item \label{point:2longblocks} if $\lfloor\alpha n^\varepsilon\rfloor\ge 2$, there exists a smallest (resp. a second smallest) long block ;
  \item \label{point:order} given a sequence of long blocks of $\mathtt{w}$, $(\zeta_{i})_{1\le i\le k}$, \emph{sorted in increasing lexicographic order}, the sequence $(\zeta_{i}\mathtt{v}_{i})_{1\le i\le k}$ is also sorted in increasing lexicographic order, for any sequence of words, $(\mathtt{v}_{i})_{1\le i\le k}$ ;
  \item \label{point:prefix1} the smallest of the long blocks is a prefix of $\mathtt{w}$ ;
  \item \label{point:prefix2} either the second smallest of the long blocks is a prefix of  the standard right factor of $\mathtt{w}$, or $\mathtt{w}\in\mathtt{a}_1\mathfrak{L}_{n-1}$ and $r_{n}(\mathtt{w})=1-\frac1n$.
\end{enumerate}
\end{prop}
\begin{proof} Item
\eqref{point:overlap} follows from point \textbf{ii.} of
Definition \ref{goodword2}. Item \eqref{point:2longblocks} follows from points
\textbf{i.}  and \textbf{iii.} of Definition \ref{goodword2}, as \textbf{iii.} insures that the prefixes $\mathtt{x}_{[1,\tilde{r}]}$ and $\mathtt{y}_{[1,\tilde{r}]}$ of two long blocks $\mathtt{x}$ and $\mathtt{y}$ are different.  That item
\eqref{point:order} holds true  follows from  \textbf{iii.} again, since \textbf{iii.}  insures not only that $\mathtt{x}$ and $\mathtt{y}$ are not tied, but also that they are not prefixes of each other. Here it can be useful to remember a basic fact about the lexicographic order: if two words $\mathtt{t}_{1}$ and $\mathtt{t}_{2}$ have prefixes, respectively $\mathtt{s}_{1}$ and $\mathtt{s}_{2}$, such that $\mathtt{s}_{1}\prec \mathtt{s}_{2}$, it does not insure that $\mathtt{t}_{1}\prec \mathtt{t}_{2}$. However, under the additional condition that $\mathtt{s}_{1}$ is not a prefix of $\mathtt{s}_{2}$, $\mathtt{s}_{1}\prec \mathtt{s}_{2}$ entails $\mathtt{t}_{1}\prec \mathtt{t}_{2}$. Thus item \eqref{point:order} fails only if some $\zeta_{i}$ is a prefix of some $\mathtt{\zeta}_{j}$, $i<j$. But this would violate point \textbf{iii.} of Definition \ref{goodword2}. As a consequence of the definition of Lyndon words, $\mathtt{w}$ begins with one of the long runs of $\mathtt{a}_{1}$ in $\langle \mathtt{w}\rangle$. This long run is a prefix of some long block due to point \textbf{i.} of Definition \ref{goodword2}. This, together with item \eqref{point:order}, entails item \eqref{point:prefix1}.

For item \eqref{point:prefix2}, consider the two smallest long blocks, $\zeta_{1}\prec\zeta_{2}$, in the necklace $\langle \mathtt{w}\rangle$, and let $k_{1}$ and $k_{2}$  be the lengths of the runs they begin with:  $\zeta_{i}=\mathtt{a}_{1}^{k_{i}}\mathtt{u}_{i}$, $i\in\ac{1,2}$, in which the words $\mathtt{u}_{i}$ do not begin with the letter $\mathtt{a}_{1}$. We know that $\mathtt{w}$ begins necessarily with $\zeta_{1}$, see the considerations leading to item \eqref{point:order}. Either the second smallest word in $\langle \mathtt{w}\rangle$, $\mathtt{w}_{2}$, begins with $\zeta_{2}$, or $\mathtt{w}_{2}$ begins with $\mathtt{a}_{1}^{k_{1}-1}\mathtt{u}_{1}$, but, since $\mathtt{a}_{1}^{k_{1}-1}\mathtt{u}_{1}$ or $\zeta_{2}$ are at least $\lceil3\log_{1/\beta}n\rceil$-letters long, they cannot be prefixes of each other, due to  point  \textbf{iii.} of Definition \ref{goodword2}.  Thus $r_{n}(\mathtt{w})=1-\frac{1}n$ if $\mathtt{a}_{1}^{k_{1}-1}\mathtt{u}_{1}\prec\zeta_{2}$, and $r_{n}(\mathtt{w})=1-\frac{v}n$ if $\mathtt{a}_{1}^{k_{1}-1}\mathtt{u}_{1}\succ \zeta_{2}$. Here $v$ denotes the position of  $\zeta_{2}$ in $\mathtt{w}$.
\end{proof}


By Proposition \ref{consGn2}, the smallest of all these
factors is $Y_0(\mathtt{w})$. Let $J_{2,n}(\mathtt{w})$ denote the index of the
second smallest factor, so that $V_{2,n}$ is given by
\begin{equation}
\label{defbegin2} V_{2,n}(\mathtt{w})=\frac1n \sum_{i=0}^{J_{2,n}(\mathtt{w})-1} |Y_i(\mathtt{w})|.
\end{equation}
If  $\mathtt{w}\in \mathtt{a}_1\mathfrak{L}_{n-1}$,
\[
r_n(\mathtt{w}) = 1-1/n,
\]
 (incidentally, we shall see later that this happens with probability $p_1+o(1)$, according to \eqref{gndeln}), while if $\mathtt{w}\in \mathcal G_{n}\cap(\mathfrak{L}_n\backslash \mathtt{a}_1\mathfrak{L}_{n-1})$, the second smallest block $Y_{J_{2,n}(\mathtt{w})}$, also a long block, is a prefix of the standard right factor, by Proposition \ref{consGn2}, and
\[
r_n(\mathtt{w}) = 1-V_{2,n}(\mathtt{w}).
\]
When $\mathtt{w}\in \mathcal G_{n}$, both cases can be detected by inspection of the two smallest blocks. 

Let $\pnh$ denote the conditional probability given $\mathcal G_n\cap\mathfrak{L}_n$:
\[\pnh\pa{A}=\frac{\Pn\pa{A\cap\mathcal G_n\cap\mathfrak{L}_n}}{\Pn\pa{\mathcal G_n\cap\mathfrak{L}_n}}=
\frac{\Pnt\pa{A\cap\mathcal G_n\cap\mathfrak{L}_n}}{\Pnt\pa{\mathcal G_n\cap\mathfrak{L}_n}}.\]

As in \cite{MR2312436}, the key point is the invariance of  $\pnh$ under uniform random permutation of the blocks $\ac{Y_1(\mathtt{w}),\dots,Y_{K_n(\mathtt{w})}(\mathtt{w})}$.
\begin{nota}
For  $\mathtt{w} \in \mathcal G_n\cap\mathfrak{L}_n$, and $\tau\in\mathfrak{S}_{K_n(\mathtt{w})}$, we set
$$\tau.\mathtt{w}=Y_0(\mathtt{w})Y_{\tau(1)}(\mathtt{w})\dots Y_{\tau(K_n(\mathtt{w}))}(\mathtt{w}),$$
and
$$\mathfrak{O}(\mathtt{w})=\{\tau.\mathtt{w}\ :\  \tau \in \mathfrak{S}_{K_n(\mathtt{w})} \}.$$
\end{nota}

\begin{prop}
\label{equivclass2}
Assume that  $\mathtt{w}\in \mathcal G_n\cap\mathfrak{L}_n$, and $\mathtt{w}'\in \mathfrak{O}(\mathtt{w})$: then  $\mathtt{w}'\in \mathcal G_n\cap\mathfrak{L}_n$ and $\mathtt{w}'$ has the same  multiset of blocks as $\mathtt{w}$ (it has the same blocks, with the same multiplicity). As a consequence, for $\mathtt{w}, \mathtt{w}' \in \mathcal G_n\cap\mathfrak{L}_n$, either $\mathfrak{O}(\mathtt{w})=\mathfrak{O}(\mathtt{w}')$ or $\mathfrak{O}(\mathtt{w})\cap \mathfrak{O}(\mathtt{w}')=\emptyset$.
\end{prop}
This follows directly from Definition \ref{goodword2} and the
definition of a code. 
Let $\mathcal C_n=\ac{\mathfrak{O}(\mathtt{w})\,;\ \mathtt{w}\in \mathcal
G_n\cap\mathfrak{L}_n}$, and let $\mathfrak C_n$ denote the
$\sigma$-algebra generated by $\mathcal C_n$. Also,
 let $X(\mathtt{w})=\pa{X_i(\mathtt{w})}_{i\ge 0}$
be the sequence of blocks of $\mathtt{w}$ sorted in increasing
lexicographic order, ended by an infinite sequence of empty words,
and  let $\Xi(\mathtt{w})=\pa{\Xi_i(\mathtt{w})}_{i\ge 0}$ be the corresponding
sequence of lengths.
\begin{corollary}
\label{condition2} The weight $p(.)$, $X$, $\Xi$, $H_n$ and $K_n$ are
$\mathfrak C_n$-measurable, and
$$
\pnh = \sum_{\mathfrak{O}\in\mathcal C_n}\frac{\mathrm{Card}(\mathfrak{O})\
p(\mathfrak{O})}{\mathbb P_n(\mathcal G_n\cap\mathfrak{L}_n)}\ \mathbb U_\mathfrak{O}.
$$
For $\mathfrak{O}\in\mathcal C_n$, given that $\mathtt{w}\in \mathfrak{O}$, the ranks of the blocks $(X_{i})_{1\le i\le
K_{n}(\mathfrak{O})}$ are uniformly distributed.
\end{corollary}
\begin{proof} The weight $p(\mathtt{w})$ depends only on the number of
letters $a_{1}$, $a_{2}$, \dots that $\mathtt{w}$ contains, not on the
order of the letters in $\mathtt{w}$, so that $p(.)$ is constant on each
$\mathfrak{O}\in\mathcal C_n$: thus, under $\pnh$, the conditional
distribution of $\mathtt{w}$ given that $\mathtt{w}\in \mathfrak{O}$ is  $\mathbb U_\mathfrak{O}$. As a
consequence of Proposition \ref{equivclass2}, $\mathcal C_n$ is a
partition of $\mathcal G_n\cap\mathfrak{L}_n$, so the relation in
Corollary \ref{condition2} is just the desintegration of $\pnh$
according to its conditional distributions given $\mathcal C_n$.
\end{proof}

\vspace{0.4cm} As in \cite[Theorem 6.5]{MR2312436}, asymptotic
independence between $\mathfrak C_n$ and $V_{2,n}$ holds under $\pnh$: for  a
$\mathfrak C_n$-measurable $\mathbb R$-valued statistic $T_n$ with
probability distribution $\chi_n$,
\begin{eqnarray}
\label{joint2} \mathcal{W}_2\pa{(T_n,V_{2,n}),\chi_n\otimes\mathbb U_1}
= \bigO{\sqrt{\frac{\log n}{n}}}.
\end{eqnarray}
In order to prove Theorem \ref{loilimit2}, let $\mu_n$ (resp.
$\tilde\mu_n$) denote the image of $\Pnt$ (resp. of $\pnh$) by
$r_n$. Set
\[
\mathfrak{L}_{n}^1 = (\mathcal G_{n}\cap\mathfrak{L}_n) \cap
a_{1}\mathfrak{L}_{n-1}=\mathcal G_{n} \cap a_{1}\mathfrak{L}_{n-1},
\quad\text{and}\quad \mathfrak{L}_{n}^2 = (\mathcal G_{n}\cap\mathfrak{L}_n)
\backslash \mathfrak{L}_{n}^1.
\]
We remark that:
\begin{itemize}
\item[\textbf{i.}]  if $\mathtt{w}\in\mathfrak{L}_{n}^1$, $r_n(\mathtt{w})=1-\frac1n$ holds true\footnote{Actually, $r_n(\mathtt{w})=1-\frac1n$ holds true if $\mathtt{w}\in a_{k}\mathfrak{L}_{n-1}(a_{k},a_{k+1},\dots, a_{n})$, but, since $\mathtt{w}\in\mathcal G_{n}$, $\mathtt{w}$ contains at least one occurrence of the letter $a_{1}$, which precludes $\mathtt{w}\in a_{k}\mathfrak{L}_{n-1}(a_{k},a_{k+1},\dots, a_{n})$ for $k\ge 2$.} ; 
 
\item[\textbf{ii.}]  if $\mathtt{w}\in\mathfrak{L}_{n}^2$, $r_n(\mathtt{w})=1-V_{2,n}(\mathtt{w})$ ;

\item[\textbf{iii.}]  when $\mathtt{w}\in \mathfrak{L}_n\backslash( \mathcal G_n\cap\mathfrak{L}_n)$, the crude bound $0\le r_n(\mathtt{w})\le 1$ will prove to be more than 
sufficient for our purposes.
\end{itemize}

First, the conditional law $\tilde \nu$, given $A$, of a bounded  r.v. $X$, defined on a probabilistic space $\Omega$, is Wasserstein-close to its unconditional law $\nu$, if $A$ is close to $\Omega$. More precisely
\begin{equation}
\label{condition}
\mathcal{W}_2\pa{\nu,\tilde \nu}
\le
2\,\pr{\Omega\backslash A}^{1/2}\,\norm{X}_\infty.
\end{equation}
As a consequence,  point \textbf{iii.}, together with Proposition \ref{Gn3}, entails that
\[
\mathcal{W}_2(\mu_n,\tilde\mu_n) = \bigO{n^{-1/2+\varepsilon}\
\log n}.
\]
Thus we shall now work on $\mathcal G_n\cap\mathfrak{L}_n$,  under $\pnh$, for $\mu_n$ has the same asymptotic behavior as $\tilde\mu_n$.

On $\mathcal G_n\cap\mathfrak{L}_n$,  we have, according to points
\textbf{i.} and \textbf{ii.},
\begin{equation*}
r_n=f_n\pa{V_{2,n},\mathbf{1}_{\mathcal
L_{n}^2}}=(1-V_{2,n})\mathbf{1}_{\mathcal
L_{n}^2}+\pa{1-\frac1n}(1-\mathbf{1}_{\mathfrak{L}_{n}^2}).
\end{equation*}
The $\mathfrak C_n$-measurability
 of $\mathfrak{L}_{n}^2$
 (see \cite[Section 7]{MR2312436} for more details) and relation (\ref{joint2}) entails asymptotic independence between $\mathbf{1}_{\mathfrak{L}_{n}^2}$ and $V_{2,n}$, and more precisely it entails that
\begin{eqnarray}
\label{EQU1}
\mathcal{W}_2\pa{(\mathbf{1}_{\mathfrak{L}_{n}^2},V_{2,n}),\chi_n\otimes\mathbb U_1}
&=&
\bigO{\sqrt{\frac{\log n}{n}}}.
\end{eqnarray}
in which $\chi_n$ denotes the probability distribution of
$\mathbf{1}_{\mathfrak{L}_{n}^2}$. Thus, there exists a probability
space, and, defined on this probability space, a couple
$(T_{n},U)$ with distribution $\chi_n\otimes\mathbb U_1$,  and a
copy of $(\mathbf{1}_{\mathfrak{L}_{n}^2},V_{2,n})$ whose
$\mathbb L^2$ distance satisfies
\[
\norm{\mathbf{1}_{\mathfrak{L}_{n}^2}-T_{n}}_{2}^2+\norm{M_{n}-U}_{2}^2
=
\bigO{\frac{\log n}{n}}.
\] 
Set
 \[
\tilde r_n=(1-U)T_n+\pa{1-\frac1n}(1-T_n).
\]
The inequality
\[
\abs{f_n\pa{d,w}-f_n\pa{d^{\prime},w^{\prime}}}^2 \le
2\pa{\abs{d-d^{\prime}}^2+\abs{w-w^{\prime}}^2},
\]
that holds for $(w,w^{\prime},d,d^{\prime})\in[0,1]^4$, entails
that
\begin{eqnarray*}
\label{EQU2}
\nonumber
\mathcal{W}_2\pa{\tilde \mu_n,\tilde r_n}
&=&
\bigO{\sqrt{\frac{\log n}{n}}}.
 \end{eqnarray*}

Finally, using an optimal coupling $\pa{T_n, \widehat T_n}$ in
which $\widehat T_n$ is a Bernoulli random variable
 with expectation $1-p_1$, independent of $U$, set
\begin{eqnarray*}
\widehat r_n & = & (1-U)\widehat T_n+\pa{1-\frac1n}(1-\widehat
T_n).
\end{eqnarray*}
As above, we obtain easily
\begin{eqnarray*}
\mathcal{W}_2\pa{\tilde r_n, \widehat r_n} &\le& \mathcal W_2\pa{T_n, \widehat T_n}
\\
&\le& \sqrt{\abs{\pnh(\mathfrak{L}_{n}^2)-(1-p_1)}}.
\end{eqnarray*}
Also
\[
(1-U)\widehat T_n+(1-\widehat T_n) =
 \widehat r_n
+\frac1n(1-\widehat T_n)
\]
 has distribution
$\mu$. Thus
\[\mathcal{W}_2\pa{\widehat r_n,\mu}\le\frac1n.
\]
Now
\begin{equation*}
{\Pn (\mathtt{a}_1 \mathfrak{L}_{n-1})- \Pn(\mathfrak{L}_{n}\backslash\mathcal G_{n})}
\le
\Pn( \mathfrak{L}_{n}^1)\le {\Pn(\mathtt{a}_1 \mathfrak{L}_{n-1})}.
\end{equation*}
So by Proposition \ref{Gn3} and the fact that $\Pn(\mathfrak{L}_{n})=
\frac 1n (1-O(\beta^{n/2}))$, we obtain
\begin{equation}
\label{gndeln}
\abs{\pnh(\mathfrak{L}_{n}^1)-p_1}=\bigO{(\log n)^2\
n^{2\varepsilon-1}}
\end{equation}
and
\[\mathcal{W}_2(\tilde r_n,\hat r_n) = \bigO{
n^{-1/2+\varepsilon}\log n}.
\]
With \eqref{EQU1}, this  yields
\[
\mathcal{W}_2(\mu_n,\mu)
=
\bigO{ n^{-1/2+\varepsilon}\log n}.\]
Since $0\le r_n\le 1$, convergence of moments follows.\hspace{4cm}$\square$


\bibliographystyle{amsalpha}
\bibliography{Lyndonfact}

\providecommand{\bysame}{\leavevmode\hbox to3em{\hrulefill}\thinspace}
\providecommand{\MR}{\relax\ifhmode\unskip\space\fi MR }
\providecommand{\MRhref}[2]{%
  \href{http://www.ams.org/mathscinet-getitem?mr=#1}{#2}
}
\providecommand{\href}[2]{#2}
\begin{thebibliography}{DMP95}

\bibitem[ABT93]{MR1230136}
Richard Arratia, A.~D. Barbour, and Simon Tavar{\'e}, \emph{On random
  polynomials over finite fields}, Math. Proc. Cambridge Philos. Soc.
  \textbf{114} (1993), no.~2, 347--368. \MR{MR1230136 (95a:60011)}

\bibitem[ABT99]{MR1702562}
\bysame, \emph{On {P}oisson-{D}irichlet limits for random decomposable
  combinatorial structures}, Combin. Probab. Comput. \textbf{8} (1999), no.~3,
  193--208. \MR{MR1702562 (2001b:60029)}

\bibitem[ABT03]{Arratia}
\bysame, \emph{\it{Logarithmic Combinatorial Structures: a probability
  approch}}, European Mathematical Society Zurich, 2003.

\bibitem[BCN05]{MR2116634}
Fr{\'e}d{\'e}rique Bassino, Julien Cl{\'e}ment, and Cyril Nicaud, \emph{The
  standard factorization of {L}yndon words: an average point of view}, Discrete
  Math. \textbf{290} (2005), no.~1, 1--25. \MR{MR2116634 (2005j:68084)}

\bibitem[BD92]{MR1161056}
Dave Bayer and Persi Diaconis, \emph{Trailing the dovetail shuffle to its
  lair}, Ann. Appl. Probab. \textbf{2} (1992), no.~2, 294--313. \MR{MR1161056
  (93d:60014)}

\bibitem[Bil99]{Billing2}
P.~Billingsley, \emph{\it{Probability and measure}}, John Wiley \& Sons, New
  York, 1999.

\bibitem[BP85]{MR797069}
Jean Berstel and Dominique Perrin, \emph{Theory of codes}, Pure and Applied
  Mathematics, vol. 117, Academic Press, Inc., Orlando, FL, 1985. \MR{797069}

\bibitem[BP07]{MR2300777}
\bysame, \emph{The origins of combinatorics on words}, European J. Combin.
  \textbf{28} (2007), no.~3, 996--1022.

\bibitem[CG96]{Conway}
J.H. Conway and R.K. Guy, \emph{\it{The book of numbers}}, Springer-Verlag,
  1996.

\bibitem[DMP95]{diaconis}
P.~Diaconis, M.J. McGrath, and J.~Pitman, \emph{\it {Riffle shuffles, cycles,
  and descents}}, Combinatorica \textbf{15, no. 1} (1995), 11--29.

\bibitem[GR93]{MR1245159}
Ira~M. Gessel and Christophe Reutenauer, \emph{Counting permutations with given
  cycle structure and descent set}, J. Combin. Theory Ser. A \textbf{64}
  (1993), no.~2, 189--215. \MR{MR1245159 (95g:05006)}

\bibitem[GS12]{DBLP:journals/corr/abs-1201-3077}
Joseph~Yossi Gil and David~Allen Scott, \emph{A bijective string sorting
  transform}, CoRR \textbf{abs/1201.3077} (2012).

\bibitem[Han93]{MR1264035}
Jennie~C. Hansen, \emph{Factorization in {${\bf F}_q[x]$} and {B}rownian
  motion}, Combin. Probab. Comput. \textbf{2} (1993), no.~3, 285--299.
  \MR{MR1264035 (95f:11056)}

\bibitem[Han94]{MR1293077}
\bysame, \emph{Order statistics for decomposable combinatorial structures},
  Random Structures Algorithms \textbf{5} (1994), no.~4, 517--533.
  \MR{MR1293077 (96f:60010)}

\bibitem[Kal97]{Kallenberg}
O.~Kallenberg, \emph{\it{Foundations of Modern Probability}}, Springer series
  in Statistics Probability and its applications, 1997.

\bibitem[Kin75]{1975}
J.~F.~C. Kingman, \emph{Random discrete distributions}, Journal of the Royal
  Statistical Society. Series B (Methodological) \textbf{37} (1975), no.~1,
  1--22.

\bibitem[Lot97]{MR1475463}
M.~Lothaire, \emph{Combinatorics on words}, Cambridge Mathematical Library,
  Cambridge University Press, Cambridge, 1997.

\bibitem[Lot02]{Lothaire1}
M.~Lothaire, \emph{\it{Algebraic Combinatorics on Words}}, vol. 90 of
  Encyclopedia of mathematics and its applications, Cambridge University Press,
  2002.

\bibitem[Lyn54]{Lyndon}
R.~Lyndon, \emph{\it {On Burnside problem I}}, Trans. American Math. Soc.
  \textbf{77} (1954), 202--215.

\bibitem[McC65]{MR2615013}
J.~W.~T. McCloskey, \emph{A model for the distribution of individuals by
  species in an environment}, ProQuest LLC, Ann Arbor, MI, 1965, Thesis--MSU.

\bibitem[MZA07]{MR2312436}
R.~Marchand and E.~Zohoorian~Azad, \emph{Limit law of the length of the
  standard right factor of a {L}yndon word}, Combin. Probab. Comput.
  \textbf{16} (2007), no.~3, 417--434. \MR{MR2312436 (2008e:68120)}

\bibitem[Oka58]{Oka}
M.~Okamoto, \emph{\it {Some inequalities related to the partial sum of binomial
  probabilities}}, Ann. Inst. Statist. Math. \textbf{10} (1958), 29--35.

\bibitem[PPY92]{MR1156448}
Mihael Perman, Jim Pitman, and Marc Yor, \emph{Size-biased sampling of
  {P}oisson point processes and excursions}, Probab. Theory Related Fields
  \textbf{92} (1992), no.~1, 21--39. \MR{MR1156448 (93d:60088)}

\bibitem[Rac91]{Rachev}
S.T. Rachev, \emph{\it{Probability Metrics and the Stability of Stochastic
  Models}}, Wiley, Chichester, U.K., 1991.

\bibitem[Reu93]{Reutenauer}
C.~Reutenauer, \emph{\it{Free lie algebras}}, Oxford Science Publications,
  1993.

\bibitem[SW09]{shorack2009empirical}
G.R. Shorack and J.A. Wellner, \emph{{Empirical processes with applications to
  statistics}}, Society for Industrial Mathematics, 2009.

\end{thebibliography}
\section{Runs statistics: proofs}
\label{appendix}
\subsection{Asymptotically almost sure properties in $\mathcal A^n$ vs $\mathcal P_n$: proof of Lemma \ref{usarg2}} This proof rephrases in probabilistic
terms some results of \cite[Section 7.1]{Reutenauer}, to which the
reader is referred for definitions. Let us define two sequences of
subsets of $\mathcal A^n$,
\begin{eqnarray*}
\mathcal A_{n,k} & = & \ac{\mathtt{w}\in\mathcal A^n\ |\ \exists \mathtt{v}\in
\mathcal A^k\text{ such that }\mathtt{w}=\mathtt{v}^{n/k}},
\\
\mathcal P_{n,k} & = & \mathcal
A_{n,k}\backslash\pa{\build{\bigcup}{1\le i< k}{}\mathcal
A_{n,i}},
\end{eqnarray*}
with probabilities $\nu_k=\Pn(\mathcal A_{n,k})$ and
$\xi_k=\Pn(\mathcal P_{n,k})$, respectively. Clearly 
\[\mathcal A_{n,n}=\mathcal A^n,\qquad \mathcal P_{n,n}=\mathcal P_n.\]
Also, if $k\vert n$, $\pa{\mathcal P_{n,i}}_{i\vert k}$  is a partition of $\mathcal A_{n,k}$ (else, both $\mathcal A_{n,k}$ and $\mathcal P_{n,k}$ are empty). Thus
\[
\nu_k = \sum_{d\vert k}\xi_d,
\]
and, by the M\"{o}bius inversion formula,
\begin{eqnarray}
\label{mobius2} \xi_k = \sum_{d\vert k}\mu(d)\nu_{k/d},
\end{eqnarray}
in which $\mu(d)$ denotes the M\"{o}bius function. On the other
hand, when $k\vert n$,
\begin{eqnarray*}
\nu_k &=& \sum_{\mathtt{w}\in\mathcal A_{n,k}}p(\mathtt{w})
\\
&=& \sum_{\mathtt{v}\in\mathcal A^k}p(\mathtt{v})^{n/k}
\\
&=& \sum_{\sum_i r_i=k}{{k}\choose{r_1,r_2,\dots}}\pa{p_1^{r_1} p_2^{r_2}\dots}^{n/k}
\\
&=& \norm{p}_{n/k}^n.
\end{eqnarray*}
Specializing \eqref{mobius2} to $k=n$, we obtain
\begin{eqnarray}
\label{mobius1} \Pn(\mathcal P_n) = \sum_{d\vert
n}\mu(d)\norm{p}_{d}^n.
\end{eqnarray}
Let the set of divisors of $n$ be $\ac{1< d_{1}< d_{2}<\dots<
d_{\ell}=n}$. Then, by \eqref{mobius1},
\begin{eqnarray*}
\abs{\Pn(\mathcal P_n)-1+\norm{p}_{d_{1}}^n} &\leq&
(\ell-1)\norm{p}_{d_{2}}^n
\\
&\leq& (n-2)\norm{p}_{d_{2}}^n,
\end{eqnarray*}
if $n$ is not prime. Else $\Pn(\mathcal
P_n)=1-\norm{p}_{d_{1}}^n$. In any case, $\abs{\Pn(\mathcal
P_n)-1+\norm{p}_{d_{1}}^n}$ is a $o\pa{\norm{p}_{d_{1}}^n}$, and,
since $d_{1}\ge 2$,
\begin{equation}
\label{boundprim2} \Pn(\mathcal P_n^c)=\bigO{\norm{p}_{2}^n}.
\end{equation}
Lemma \ref{usarg2} is a direct consequence of
\[
\Pnt(A) = \frac{\P_{n}({\pi}^{-1}(A))}{\P_{n}(\mathcal P_n)}\,,
\]
and of \eqref{boundprim2}.

\subsection{Alternative representations for $\Pn$}
The asymptotic behaviour of the factorizations of $n$-letters general random words is predicted by the lengths and positions of runs of the letter $\mathtt{a}_{1}$, provided that these lengths and positions satisfy a set of properties that hold true, but for a vanishing probability as $n$ grows, for what we call \emph{good words} (see Definition \ref{goodword2}). Thus, a random word  $\mathtt{w}\in \mathcal A^n$ is a good word, or not, depending on  $\varphi_{n}(\mathtt{w})$. The proof that the probability of bad words vanishes relies on two descriptions, given in this section, of the probability distribution $\Bn$ of $\varphi_{n}(\mathtt{w})$ under $\Pn$. 

Under $\Pn$, $\varphi_{n}$ is a $\mathcal B^n$-valued random variable, a random word, a sequence of $n$ independant symbols, each of them being a $\mathtt{0}$ with probability $p_{1}$, a $\mathtt{1}$ with probability $1-p_{1}$ : $\Bn$ denote the probability distribution of $\varphi_{n}$, i.e. the push-forward of $\Pn$ by $\varphi_{n}$. For the next proofs, however, $\Bn$ shall be seen as the push-forward of two probability measures on the set $\mathcal B^{\N}$ of infinite words, by the truncation operation  $\psi_{n}$  defined, for $\omega\in \mathcal B^{\N}$, by:
$$\omega = \omega_1\omega_2\omega_3\dots\quad\longrightarrow\quad\psi_{n}(\omega)=\omega_{[1,n]}.$$
First, $\Bn$ is the probability distribution of $\psi_{n}$ under the product measure $$\B=\left(p_{1}\delta_{\mathtt{0}}+(1-p_{1})\delta_{\mathtt{1}}\right)^{\otimes\N}.$$ 
Next, let $\eta=(\eta_{n})_{n\ge1}$ (resp. $\theta=(\theta_{n})_{n\ge1}$) be a sequence of independent  geometric random variables with expectation $(1-p_1)^{-1} $ (resp. with expectation ${p_1}^{-1}$),  defined on some probability space $(\Omega,\mathcal{F},\mathbb{P})$,  let $\xi$ be a Bernoulli random variable with parameter $1-p_1$, and assume that $\xi$, $\eta$ and $\theta$ are independent.
For $m\ge 1$, set:
\begin{align*}
S_{m}&=\sum_{k=1}^{m}(\eta_{k}+\theta_{k}),
\end{align*}
and consider the infinite random word
$$ \Upsilon\ =\ 
\begin{cases}
\mathtt{0}^{\eta_{1}}\mathtt{1}^{\theta_{1}}\mathtt{0}^{\eta_{2}}\mathtt{1}^{\theta_{2}}\dots&\text{if~}\xi=0,
\\
\mathtt{1}^{\theta_{1}}\mathtt{0}^{\eta_{1}}\mathtt{1}^{\theta_{2}}\mathtt{0}^{\eta_{2}}\dots&\text{if~}\xi=1,
\end{cases}$$
that is,  $ \Upsilon$ is defined by the sequences $\eta$ and $\theta$ of  its runs' lengths. Then

\begin{prop}
\label{ident}
The probability distribution of $\Upsilon$ is $\B$. As a consequence, $\Upsilon_{[1,n]}$, $\psi_{n}$ and $\varphi_{n}$ have the same distribution $\Bn$.
\end{prop}
\begin{proof}  We already know that, for all $n$, $\psi_{n}$ and $\varphi_{n}$ have the same distribution $\Bn$. We need to prove that $\Bn$ is also the distribution of $\Upsilon_{[1,n]}$ : for any $\ell\ge1$, and any finite word $\mathtt{w}\in \mathcal{B}^{\ell}$, for instance of the form $\mathtt{0v10}$, i.e.  having an even number of runs, say $2m$, followed by $\mathtt{0}$, the first run being thus a run of $\mathtt{0}$s, 
we can write
\begin{align*}
\mathtt{w}&=\mathtt{0}^{k_{1}}\mathtt{1}^{\ell_{1}}\mathtt{0}^{k_{2}}\mathtt{1}^{\ell_{2}}\dots\mathtt{0}^{k_{m}}\mathtt{1}^{\ell_{m}}\mathtt{0},
\\
\ell-1&=s_{m}=\sum_{i=1}^{m}(k_{i}+\ell_{i}),
\end{align*}
and we have
\begin{align*}
\mathbb{B}_{\ell}\left(\ac{\mathtt{w}}\right)
&=p_{1}^{1+k_{1}+\dots+k_{m}}(1-p_{1})^{\ell_{1}+\dots+\ell_{m}}
\\
&=p_{1}\prod_{i=1}^{m}p_{1}^{k_{i}-1}(1-p_{1})\prod_{i=1}^{m}(1-p_{1})^{\ell_{i}-1}p_{1}
\\
&=\mathbb{P}\left(\xi=0\right)\prod_{i=1}^{m}\mathbb{P}\left(\eta_{i}=k_{i},\theta_{i}=\ell_{i}\right)
\\
&=\mathbb{P}\left(\Upsilon_{[1,\ell]}=\mathtt{w}\right).
\end{align*}
For $\xi=1$, and for $\mathtt{w}$ of the form $\mathtt{1v01}$, or even when $\mathtt{w}$ does not end with the beginning of a new run, the computation is similar, in the last case using $\mathbb{P}\left(\theta_{i}> k_{i}\right)=(1-p_{1})^{k_{i}}$, for instance. This also entails that the probability distribution of $\Upsilon$ is $\B$. 
\end{proof}

\subsection{Number of runs: proof of Lemma \ref{Nn2}}

\begin{prop}
\label{bound1}
$\mathbb{P}_{n}\left(N_{n}^{(\mathtt{a}_{1})}< m\right)\le\mathbb{P}\left(S_{m}> n\right)$.
\end{prop}
\begin{proof}
If $N_{n}^{(\mathtt{0})}\circ\psi_{n}(\omega)< m$,  the $m$th run
  of $\mathtt{0}$s  of $\psi(\omega)$ begins after its $n$th letter. According to Proposition \ref{ident}, this last event has the same probability for   $\omega$ or for the infinite random word $\Upsilon$, but the  $m$th run   of $\mathtt{0}$s  of $\Upsilon$ begins either at the position  $1+S_{m-1}(< S_{m})$, or at the position  $\theta_{m}+1+S_{m-1}(\le S_{m})$, according to the value of $\xi$. \end{proof}

Thus, by Chebyshev's inequality,
\begin{align*}
\mathbb{P}_{n}\left(N_{n}^{(\mathtt{a}_{1})}< m\right)&\le\mathbb{P}\left(S_{m}> n\right)
\\
&\leq  \frac{\text{Var}(S_m)}{\pa{n-m\esp{}{\eta_k+ \theta_{k}}}^2}\,.
\end{align*}
Since $\esp{}{\eta_k+ \theta_{k}}=\sigma^{-2}$, with the choice $m=a n+b$, $b\in\R$, $a< \sigma^{2}$, we obtain
\begin{align}
\label{necklacerun} 
\Pn\pa{N_n^{(\mathtt{a}_{1})}< a n +b }
&=
\bigO{n^{-1}}.
\end{align}
For a primitive word $\mathtt{w}$,  $N_n^{(\mathtt{a}_{1})}(\mathtt{w}) \leq
N_n^{(\mathtt{a}_{1})}(\pi(\mathtt{w}))+1$,  thus
\begin{align*}
\Pn\pa{N_n^{(\mathtt{a}_{1})}\circ\pi< a n +b }\le
\Pn\pa{N_n^{(\mathtt{a}_{1})}< a n +b +1}.
\end{align*}
With that in view, Lemma~\ref{usarg2} extends \eqref{necklacerun} to $\Pnt$. Note that, with some additional work, one obtains easily, for any  $\varepsilon>0$,
$$\Pn\pa{N_n^{(\mathtt{a}_{1})}< \sigma^{2}(1-\varepsilon)n } =
\bigO{e^{-\eta n}},$$ 
for a suitable $\eta>0$. However, the weaker Lemma \ref{Nn2} suits our aims here.

\subsection{Number of long runs : proof of Lemma \ref{lemRn2}} 
Given that the probability of a run of $\texttt{a}_{1}$ longer than $(1-\varepsilon)\log_{1/p_1}{n}$ is approximately $n^{\varepsilon}/n$, an (admittedly flawed) argument suggests that in an $n$-letters long random word, there are many (that is, $\Theta\pa{n^{\varepsilon}}$) runs longer than $(1-\varepsilon)\log_{1/p_1}{n}$. For a more precise and concise argument, let $\hat{N}_{n}=N_{n}^{(\mathtt{0})}\left(\Upsilon_{[1,n]}\right)$ (resp. $\hat{H}_n$) denote the number of runs of $\mathtt{0}$'s in the word $\Upsilon_{[1,n]}$ (resp. the number  of runs with length at least $(1-\varepsilon)\log_{1/{p_1}} n$). 
Each  run of the letter $\texttt{a}_1$ in the word $\texttt{w}$ is matched with a run of $\texttt{0}$ of the same length in the word $\varphi(\texttt{w})$, thus, according to Proposition \ref{ident},  $(N_n^{(\mathtt{a}_{1})},H_n)$ (defined on $(\mathcal{A}^{n},\Pn)$) and $(\hat{N}_{n},\hat{H}_n)$ (defined on  $(\Omega,\mathcal{F},\mathbb{P})$) have the same probability distribution. 

We let, for $i \geq 1$,
\[
B_{i} = \mathbf{1}_{\{\eta_{i}\geq(1-\varepsilon)\log_{1/
p_1}{n}\}},\quad \hat{S}_{m}=\sum_{i=1}^{m}B_{i}.
\]
The sequence of lengths of  runs of $\mathtt{0}$ in $\Upsilon_{[1,n]}$ differs from $(\eta_{i})_{1 \le i\le \hat{N}_n}$, possibly, only at the last term,  due to the truncation of $\Upsilon$. As a consequence, 
\begin{equation}
\label{2q} \hat{H}_{n}\geq \hat{S}_{\hat{N}_n-1}.
\end{equation}
Note also that, under $\mathbb P$,  $(B_i)_{i \geq 1}$ is a
Bernoulli process, and that its parameter $p(n, \varepsilon)$
satisfies $n^{\varepsilon-1} \leq p(n,\varepsilon) \leq
n^{\varepsilon-1}/p_{1}$. Thus relation \eqref{2q}, with Lemma \ref{Nn2}, entails that, under $\P$, $\hat{H}_{n}$ is, roughly speaking,
stochastically larger than the binomial distribution with parameters $an+b$ and $p(n,\varepsilon)$, provided that $a<\sigma^{2}$. More precisely, \begin{align*}
\P\pa{\hat{H}_{n} < \alpha  n^\varepsilon}
&\le \P\pa{\hat{N}_{n}< an+2 }
+\P\pa{\hat{N}_{n}\ge an+2\text{~and~} \hat{S}_{\hat{N}_n-1} < \alpha  n^\varepsilon}
\\
&\le\P\pa{\hat{N}_{n}< an+2 }
+\P\pa{\hat{S}_{\lceil an\rceil} < \alpha  n^\varepsilon} .
\end{align*}
Lemma \ref{Nn2} takes care of the first term on the right hand side. For the second term,  by Okamoto's inequality \cite[Th. 2(ii)]{Oka}, a binomial random variable $S_{n,p}$ with parameters $n$ and $p<1/2$ satisfies :
\begin{equation*}
\P\pa{ S_{n,p}-pn\le -cn}
 <
\exp \pa{-nc^2/(2pq)}.
\end{equation*}
As a consequence
\begin{align*}
\P\pa{\hat{S}_{\lceil an\rceil} < \alpha  n^\varepsilon}
&\le
\P\pa{\hat{S}_{\lceil an\rceil}-\lceil an\rceil p(n,\varepsilon)< \alpha  n^\varepsilon-\lceil an\rceil p(n,\varepsilon)}
\\
&\le
\P\pa{\hat{S}_{\lceil an\rceil}-\lceil an\rceil p(n,\varepsilon)< (\alpha -a) n^\varepsilon},
\end{align*}
and, for $\alpha<a$ and $\lceil an\rceil\le 2an$, Okamoto's inequality entails that
\begin{align*}
\P\pa{\hat{S}_{\lceil an\rceil} < \alpha  n^\varepsilon}
&\le
\exp\pa{-\frac{(a-\alpha)^{2}p_{1}}{4a}  n^\varepsilon}.
\end{align*}
The first statement of  Lemma \ref{lemRn2} follows.
For the proof of the second statement, we note that if
$\mathtt{w}$ is a primitive word,
\begin{equation}
\label{picirc2} H_{n} \circ \pi(\mathtt{w}) \geq
H_{n}(\mathtt{w})-1,
\end{equation}
with equality when $\mathtt{w}$ begins and ends with long runs. Together
with Lemma \ref{usarg2}, it entails that
\begin{eqnarray*}
\Pnt\pa{H_{n} < \alpha  n^\varepsilon} &\le&
\Pn\pa{\ac{\mathtt{w}\in\mathcal P_n , H_{n}\circ \pi(\mathtt{w})<
\alpha n^\varepsilon}} +\bigO{ \beta^{n/2}}
\\
&\le& \Pn\pa{H_{n}  < \alpha  n^\varepsilon+1}+\bigO{
\beta^{n/2}}.
\end{eqnarray*}
and the Lemma follows since, as above,  $\P\pa{\hat{S}_{\lceil an\rceil} < 1+\alpha  n^\varepsilon}=\bigO{ n^{-1}}$.

\subsection{Large values of the longest runs : proof of Lemma \ref{lemMnGV2}} 
Recall that $M^{(\mathtt{0})}_n(\varphi(\mathtt{w}))$ (resp. $M^{(\mathtt{1})}_n(\varphi(\mathtt{w}))$) denote the length of the largest runs of the letter $\mathtt{a}_{1}$ (resp. non-$\mathtt{a}_{1}$ letters) of some word $\mathtt{w}\in\mathcal{A}^{n}$, see Definition \ref{defdebase}, and set 
\[
A_{1,n}=\ac{M^{(\mathtt{1})}_n\circ\varphi_{n} \ge 2\log_{1/(1-{p_1})}{n}},\quad A_{0,n}=\ac{M^{(\mathtt{0})}_n\circ\varphi_{n} \ge 2\log_{1/{p_1}}{n}}.
\]
In order to prove that $\Pn(A_{i,n})$ or $\Pnt(A_{i,n})$ are  $\bigO{n^{-1}}$, we use again  Proposition  \ref{ident} then Lemma \ref{usarg2} : for any $\mathtt{i}\in\{\mathtt{0},\mathtt{1}\}$, let $\hat{M}^{(\mathtt{i})}_n$ denote the length of the largest run of $\mathtt{i}$'s of the word $\Upsilon_{[1,n]}$, so that, by Proposition \ref{ident},  $(\hat{M}^{(\mathtt{0})}_n,\hat{M}^{(\mathtt{1})}_n)$ has the same probability distribution as $(M^{(\mathtt{0})}_n\circ\varphi_{n},M^{(\mathtt{1})}_n\circ\varphi_{n})$. As a consequence, for $y>0$, we have:
\begin{align*}
\Pn(M^{(\mathtt{0})}_n\circ\varphi_{n} \leq y ) &= \P(\hat{M}^{(\mathtt{0})}_n \leq y )
\\
&\ge \P(\forall i \in \{1, \dots,n\}, \;
\eta_i \leq y)
\\
&\ge \pa{1-p_1^{ \lfloor y \rfloor}}^n,
\end{align*}
the first inequality due to $\hat{N}_n\le n$. Choosing $y=\ceil{2\log_{1/p_1}n}-1$,
we obtain that
\begin{equation*}
\Pn\pa{A_{0,n}}
=
\bigO{n^{-1}}.
\end{equation*}

Note that for a primitive word $\mathtt{w}$, $M^{(\mathtt{0})}_n\circ \varphi_{n}\circ \pi(\mathtt{w})$ differs from $M^{(\mathtt{0})}_n\circ \varphi_{n}(\mathtt{w})$ only if the word $\mathtt{w}$ begins \emph{and} ends with the letter $\mathtt{a}_{1}$. More precisely, we have, according to Definition \ref{defdebase}, 
\begin{eqnarray*}
M^{(\mathtt{0})}_n\circ \varphi_{n}\circ \pi(\mathtt{w}) &=&
\max\{M^{(\mathtt{0})}_n\circ \varphi_{n}(\mathtt{w}),\pa{W_1(\mathtt{w})+W_{N_n}\circ \varphi_{n}(\mathtt{w})}\ind_{\mathtt{w}_{1}= \mathtt{a}_{1}= \mathtt{w}_{n}}\}
\\
&\leq&
\max\ac{M^{(\mathtt{0})}_n(\mathtt{w}),\pa{W_1(\mathtt{w})\ind_{\mathtt{w}_{1}= \mathtt{a}_{1}}+W_{N_n}(\mathtt{w})\ind_{\mathtt{w}_{n}= \mathtt{a}_{1}}}},
\end{eqnarray*}
Since $\Pn$ is invariant under words' reversal, $\pa{W_1,\mathtt{w}_{1}}$
and $\pa{W_{N_n},\mathtt{w}_{n}}$ have the same probability distribution. Thus, from Lemma
\ref{usarg2}, we deduce that
$$\Pnt\left(A_{0,n}\right)
\leq 2\,\Pn \left(W_1 \ind_{ \mathtt{w}_{1}= \mathtt{a}_{1}}\ge \log_{1/p_1}{n}\right)+\Pn\pa{A_{0,n}}+\bigO{\norm{p}_{2}^n}.
$$
which leads to the desired bound for $\Pnt\left(A_{0,n}\right)$, since, for $1\le k\le n$, $$\Pn \left(W_1 \ind_{ \mathtt{w}_{1}= \mathtt{a}_{1}}\ge k\right)=p_{1}^{k}.$$ Similar arguments hold for $\Pn\left(A_{1,n}\right)$ and $\Pnt\left(A_{1,n}\right)$.

\subsection{$\Lambda_{n}$ is large : proof of Proposition \ref{nombrecord}}
\label{secnombrecord}
Proposition \ref{nombrecord} asserts that with a probability close to 1, $\Lambda_{n}$ is at least of order $\log n$. By invariance of $\Pnb$ under uniform random permutations of the blocks $Y_{i}$, the sequence of ranks of the long blocks is, conditionally given that $H_n=k$,  a random uniform permutation of $\mathfrak{S}_k$.  Thus the conditional distribution of the number $\Lambda_{n}$ of Lyndon factors obtained this way, given that $H_n=k$, has the same law as the number of records (or of cycles) of a uniform random permutation in $\mathfrak{S}_k$  (see \cite[Ch. 1]{Arratia} or \cite[Ch. 11]{Lothaire1}), with generating function 
\[\frac1{k!}\ x(x+1)(x+2)\dots(x+k-1)=\frac1{k!}\sum_{0\le j\le k}\left[{k \atop j}\right]x^j,\]
in which $\left[{k \atop j}\right]$ is a Stirling number of the first kind.
We can thus describe the conditional law of $\Lambda_{n}$ as follows : consider a sequence $B=(B_i)_{i\ge 1}$ of independent Bernoulli random variables with respective parameters $1/i$, $B$ and $H_n$ being independent. Set
$$\tilde{S}_{n}=\sum_{1\le i\le n} B_i$$
and
$$\tilde\Lambda_{n}=\tilde{S}_{H_n} = \sum_iB_i\ 1\!\!1_{1\le i\le H_n}.$$
Then $\Lambda_{n}$ and $\tilde\Lambda_{n}$ have the same distribution, and we shall use the notation  $\Lambda_{n}$ for both of them.

The $m$-th harmonic number has the asymptotic expansion
$$\sum_{i=1}^{m}{1/i}\,=\,\mathfrak{H}_m\,=\,\ln m+\gamma+\frac{1}{2{m}}-\frac{1}{12{m}^2}+\dots,$$
in which $\gamma$ is the Euler-Mascheroni constant (see \cite{Conway}). We have
$$
\mathbb E(\tilde{S}_{n})=\mathfrak{H}_{n}
\hspace{.6cm}\text{and}\hspace{0.6cm}
\text{Var}(\tilde{S}_{n})=\mathfrak{H}_{n}-\sum_{i=1}^{n}{\frac{1}{i^2}}.
$$
By Lemma \ref{lemRn2} :
\begin{eqnarray*}
\mathbb P_n\pa{\Lambda_{n}< \varepsilon\log n/3}
&\leq&
\mathbb{P}_n\pa{\Lambda_{n}< \varepsilon\log n/3\ |\ H_n\geq\alpha n^\varepsilon}+ \bigO{ n^{-1} }.
\end{eqnarray*}
In addition
\begin{eqnarray*}
\mathbb{P}_n\pa{\Lambda_{n}< \varepsilon\log n/3\ |\ H_n\geq\alpha n^\varepsilon} &\leq& \mathbb{P}_n\pa{\tilde{S}_{\alpha n^\varepsilon}<\varepsilon\log n/3}
\\
&\leq&
\mathbb P_n\left\{\left|\tilde{S}_{\alpha n^\varepsilon}-\mathbb E(\tilde{S}_{\alpha n^\varepsilon}) \right|\geq \varepsilon\log n/2\right\}
\\
&=&
\bigO{ \frac 1{\log n} },
\end{eqnarray*}
in which the second inequality holds true provided that
\[\mathbb E(\tilde{S}_{\alpha n^\varepsilon})-\varepsilon\log n/3\geq \varepsilon\log n/2,\]
i.e. for $n$ large enough, and the last equality follows from the Bienaym\'e-Chebyshev inequality and from the asymptotic behavior of $\mathfrak{H}_{n}$.

\end{document}